\documentclass[12pt]{amsart}

\usepackage{amssymb, amsmath}
\usepackage{amsthm}
\usepackage{mathrsfs}
\usepackage[pdfborder={0 0 0 [0 0 ]},bookmarksdepth=3, bookmarksopen=true,unicode=true]{hyperref}
\usepackage{cite}
\usepackage{enumitem}
\usepackage{cases}
\usepackage[font={it}, margin=1cm]{caption}
\usepackage{verbatim}
\usepackage[capitalize]{cleveref}
\usepackage{color}
\usepackage[ps,all,arc,rotate]{xy}
\usepackage{graphicx}

\newtheorem{theorem}{Theorem}[section]
\newtheorem{question}[theorem]{Question}
\newtheorem{definition}[theorem]{Definition}
\newtheorem{lemma}[theorem]{Lemma} 
\newtheorem{proposition}[theorem]{Proposition} 
\newtheorem{thmletter}{Theorem}

\newtheorem{corollary}[theorem]{Corollary} 
\newtheorem{corolletter}[thmletter]{Corollary} 

\theoremstyle{definition}
\newtheorem{example}[theorem]{Example}
\newtheorem{rmk}[theorem]{Remark}

\numberwithin{figure}{section}

\makeatletter
\let\c@figure\c@theorem
\makeatother

\newcommand{\p}[1]{\noindent {\newline\bf #1.}}

\newcommand{\aut}{\operatorname{Aut}}
\newcommand{\BS}{\operatorname{BS}}

\newcounter{dawidcomments}

\newcounter{alancomments}

\newcounter{gilescomments}

\newcounter{nameOfYourChoice}

\title{Algebraically hyperbolic groups}
\author{Giles Gardam}
\address{
Mathematisches Institut, Universit\"at Bonn, Endenicher Allee 60, 53115 Bonn,
Germany}
\email{gardam@math.uni-bonn.de}

\author{Dawid Kielak}
\address{Mathematical Institute, University of Oxford, Andrew Wiles Building,
Radcliffe Observatory Quarter,
Woodstock Road,
Oxford
OX2 6GG,
United Kingdom}
\email{kielak@maths.ox.ac.uk}

\author{Alan D. Logan}
\address{
Heriot-Watt University, Edinburgh, EH14 4AS,
UK}
\email{A.Logan@hw.ac.uk}


\begin{document}
\maketitle

\begin{abstract}
We initiate the study of torsion-free algebraically hyperbolic groups; these groups generalise torsion-free hyperbolic groups and are intricately related to groups with no Baumslag--Solitar subgroups.
Indeed, for groups of cohomological dimension $2$ we prove that algebraic hyperbolicity is equivalent to containing no Baumslag--Solitar subgroups.
This links algebraically hyperbolic groups to two famous questions of Gromov; recent work has shown these questions to have negative answers in general, but they remain open for groups of cohomological dimension $2$.

We also prove that algebraically hyperbolic groups are CSA, and so have canonical abelian JSJ-decompositions. In the two-generated case we give a precise description of the form of these decompositions.
\end{abstract}

\section{Introduction}
\label{introduction}
The notion of a hyperbolic group, introduced by Gromov in his famous essay \cite{Gromov1987Essay}, embodies a paradigm shift in the theory of infinite discrete groups which brought geometric ideas to the forefront.
Hyperbolic groups can be also defined using more combinatorial terms (the Dehn function), but it is less clear if they can be characterised by their group theoretic or topological properties, like subgroup structures and existence of finite classifying spaces.
We are motivated by two questions in this vein.
A group is \emph{$\BS$-free} if it contains no Baumslag--Solitar subgroups $\BS(m, n):=\langle a, t\mid t^{-1}a^mt=a^n\rangle$.
Torsion-free hyperbolic groups are $\BS$-free and have finite classifying spaces.
\begin{question}
\label{Question:1}
Let $G$ be a group with a finite classifying space of dimension at most $3$. If $G$ is $\BS$-free, is it hyperbolic?
\end{question}
Let us introduce the following notion: a group $G$ is \emph{weakly algebraically hyperbolic (weakly AH)} if it contains no $\mathbb{Z}^2$ subgroups, and for every non-trivial $x\in G$ if the elements $x^m$ and $x^n$ are conjugate then $|m|=|n|$.
Torsion-free hyperbolic groups are weakly AH.
\begin{question}
\label{Question:2}
Let $G$ be a group with a finite classifying space of dimension at most $3$. If $G$ is weakly AH, is it hyperbolic?
\end{question}
Variations on these two questions have driven research in Geometric Group Theory since the 1990s.
In particular, Gersten asked if every $\BS$-free one-relator group is hyperbolic \cite[Remark, p. 734]{Allcock1999homological}, which reduces to a sub-question of Question \ref{Question:1}, while both Questions \ref{Question:1} and \ref{Question:2} have positive answers for $3$-manifold groups, by Perelman's Geometrisation theorem, for free-by-cyclic groups \cite{Brinkmann2000hyperbolic}, for ascending HNN-extensions of free groups \cite{Mutanguha2021dynamics}, and for fundamental groups of special cube complexes \cite{Caprace2009geometric} (without restriction on the dimension).

A group is of \emph{type ${\mathtt{F}}$} if it has a finite classifying space, while the \emph{geometric dimension} of a group $G$ is the minimal dimension of a classifying space for $G$.
Therefore, the above questions are for groups that are of type $\mathtt{F}$ and have geometric dimension at most $3$ (and indeed, one can show conversely using Schanuel's lemma that a group has finite classifying spaces of dimension at most $3$ if it satisfies the two conditions).
These two conditions are necessary.
Brady constructed a finitely presented, non-hyperbolic group that embeds into a hyperbolic group \cite{Brady1999Branched} (see also \cite{Lodha2018Hyperbolic,Kropholler2021subgroups,Llosa2021hyperbolic,Llosa2024hyperbolic}).
This example is both weakly AH and $\BS$-free, but is not of type $\mathtt{F}$.

The versions of Questions \ref{Question:1} and \ref{Question:2} attributed to Gromov \cite[Questions 1.1 \& 1.2]{Bestvina2004questions}, since they appeared in print after Brady's example, are formulated within the realm of groups of type $\mathtt F$.
Recently, Italiano, Martelli and Migliorini constructed a non-hyperbolic group $G$ of type $\mathtt F$ that embeds into a hyperbolic group \cite[Corollary 2]{Italiano2023hyperbolic}, which is therefore a counter-example to both of Gromov's questions.
The group $G$ they construct has geometric and cohomological dimension $4$ (note that every group of cohomological dimension $n \geqslant 3$ has geometric dimension $n$ as well), and so we see that the bound on geometric dimension in Questions \ref{Question:1} and \ref{Question:2} is also necessary.
Therefore, Gromov's general questions do not hold, but Questions \ref{Question:1} and \ref{Question:2} remain open.

\p{Algebraically hyperbolic groups}
Motivated by Questions \ref{Question:1} and \ref{Question:2}, and by the fact that the classes of type $\mathtt F$ groups that are either $\BS$-free or weakly AH contain the class of torsion-free hyperbolic groups properly (by Italiano--Martelli--Migliorini's result), and so deserve to be studied in their own right, we introduce and study the class of ``torsion-free algebraically hyperbolic'' (AH) groups.
(All the groups studied in this paper are torsion-free, for example groups of type $\mathtt F$ are torsion-free, so we often omit the ``torsion-free'' phrasing.)
Recall that a group $G$ is \emph{commutative-transitive} if commutativity is a transitive relation on the non-identity elements of $G$ (so for $g, h, k\neq1$, if $[g,h]=1$ and $[h, k]=1$ then $[g, k]=1$).
\begin{definition}
\label{def:ah}
A torsion-free group $G$ is \emph{algebraically hyperbolic (AH)} if
\begin{enumerate}[label=(\alph*)]
\item\label{AHdef:CT} $G$ is commutative-transitive, and
\item\label{AHdef:BS} $G$ contains no subgroup of the form $H\rtimes\mathbb{Z}$ with $H$ non-trivial locally cyclic.
\setcounter{nameOfYourChoice}{\value{enumi}}
\end{enumerate}
\end{definition}

Basic properties and examples of AH groups are given in \cref{sec:definition}.
We later prove the following.

\begin{thmletter}[Theorem \ref{thm:inclusions}]
\label{thm:inclusions:intro}
We have the following chain of proper inclusions of classes of groups:
\[\text{AH}\subset\text{weakly AH}\subset\text{$\BS$-free}.\]
\end{thmletter}
Our focus is on proving hyperbolic-like properties, thus giving evidence towards Questions \ref{Question:1} and \ref{Question:2}, and on understanding the connections between the classes, as the constructions of Italiano--Martelli--Migliorini do not prove any relationship between the classes of $\BS$-free and weakly AH groups of type $\mathtt F$, only that each properly contains the class of torsion-free hyperbolic groups.

\p{AH vs weakly AH vs {\boldmath $\BS$}-free}
Despite the properness of the inclusions $\text{AH}\subseteq\text{weakly AH}\subseteq\text{$\BS$-free}$ in general,
under certain natural conditions they {are} equivalent.

Our next theorem states that all three definitions coincide for groups of cohomological dimension $2$.
The class of groups of cohomological dimension $2$ sits between the classes of groups of geometric dimension $2$ and geometric dimension at most $3$ (the Eilenberg--Ganea Conjecture deals with the ``properness'' of the first of these inclusions).
Torsion-free one-relator groups, as well as many of the above-mentioned classes of groups with positive answers to Questions \ref{Question:1} and \ref{Question:2}, have cohomological dimension $2$.

\begin{thmletter}[Theorem \ref{thm:coHom2Main}]
\label{thm:coHom2Main:intro}
Let $G$ be a group of cohomological dimension $2$. Then the following are equivalent.
\begin{enumerate}
\item $G$ is AH.
\item $G$ is weakly AH.
\item $G$ is $\BS$-free.
\end{enumerate}
\end{thmletter}

This means that for groups of cohomological dimension $2$, Questions \ref{Question:1} and \ref{Question:2} are equivalent.
Dimension $2$ is sharp here, as there exist $\BS$-free groups of cohomological, and hence geometric, dimension $3$ which are not weakly AH (Theorem \ref{thm:coHom3Main}).

The discussion so far has focused on topological properties (cohomological dimension, classifying spaces), but we now turn to two algebraic properties.
A group $G$ is \emph{locally indicable} if every non-trivial finitely generated subgroup $K$ of $G$ surjects onto the infinite cyclic group, $K\twoheadrightarrow \mathbb{Z}$.
Examples of such groups include torsion-free one-relator groups, and more generally one-relator products of locally indicable groups \cite{Howie1982Locally}, and bi-ordered groups \cite{Rhemtulla2002suspect}.
A group $G$ is \emph{residually finite} if the intersection of all its finite-index subgroups is trivial.
Examples of such groups include finitely generated linear groups \cite{mal1940faithful} and finitely generated free-by-cyclic groups \cite{Baumslag1971Finitely}.
It is a famous conjecture that all hyperbolic groups are residually finite (see for example \cite{Kapovich2000equivalence}).

\begin{thmletter}[Theorem \ref{thm:LI}]
\label{thm:LI:intro}
Let $G$ be a locally indicable group or a torsion-free residually finite group.
Then $G$ is weakly AH if and only if $G$ is AH.
\end{thmletter}

Any group which is weakly AH but not AH must contain a finitely generated subgroup $K$ which is a cyclic extension of an infinite torsion group (see Theorem \ref{thm:wAHclassification}).
Therefore, proving that there exists a weakly AH group of type $\mathtt F$ which contains such a subgroup $K$ would provide a new, stronger counter-example to Gromov's version of Question \ref{Question:2}, mentioned above, in the sense that the counter-example would not embed into a hyperbolic group.
We however suspect that no weakly AH group of type $\mathtt F$ contains such a subgroup, as this would follow from a positive answer to the following question.

\begin{question}
\label{qn:WAHvsAH}
Let $K$ be a finitely generated cyclic extension of an infinite torsion group.
Is the cohomological dimension of  $K$ necessarily infinite?
\end{question}

If \cref{qn:WAHvsAH} has a positive answer, then a group of type $\mathtt F$ is weakly AH if and only if it is AH (see Corollary \ref{corol:AnswerQn}).

Suppose that $K$ above is the example constructed by Adian~\cite{Adian1971Certain}, that is a central extension of a free Burnside group on at least $2$ generators of odd exponent greater than $666$.
Using the Lyndon--Hochschild--Serre spectral sequence and a result of Mislin--Talelli \cite[Lemma 4.12(ii)]{MislinTalelli2000} one easily sees that  $H^n(K;\mathbb Z K)$ is trivial for all $n>3$. As is shown below, the cohomological dimension of $K$ cannot be less than $3$, and so the cohomological dimension of Adian's $K$ is either $3$ or $\infty$.

\p{The CSA property}
We prove that AH groups satisfy a strong hyperbolic-like property:
A subgroup $M$ of a group $G$ is \emph{malnormal} if $M^g\cap M=1$ for all $g\in G\smallsetminus M$.
A group $G$ is \emph{CSA} (for \emph{conjugacy separated abelian}) if all maximal abelian subgroups of $G$ are malnormal in $G$.
Torsion-free hyperbolic groups are CSA, and this property played a central role in the resolution of Tarski's problem by Kharlampovich--Myasnikov and Sela.
Guirardel--Levitt have recently developed a theory of JSJ-decompositions for finitely generated CSA groups \cite{Guirardel2017JSJ};
JSJ-decompositions are a powerful tool in Geometric Group Theory, underlying for example the aforementioned proofs of Tarski's problem as well as the resolution of the isomorphism problem for hyperbolic groups.

\begin{thmletter}[Theorem \ref{thm:CSA}]
\label{thm:CSAMain:intro}
Torsion-free AH groups are CSA.
\end{thmletter}

In Section \ref{sec:splittings}, we build on this result to classify the abelian JSJ-decompositions of $2$-generated AH groups.
This classification plays a foundational role in a subsequent paper of the authors, linking JSJ decompositions and ``Friedl--Tillmann polytopes'' (a group-theoretic analogue of the Thurston norm) for certain $\BS$-free one-relator groups \cite{gardam2021jsj}.

Gildenhuys, Kharlampovich and Myasnikov proved that $\BS$-free tor\-sion-free one-relator groups are CSA \cite[Theorem 7]{Gildenhuys1995CSA}.
Torsion-free one-relator groups have cohomological dimension $2$, and so Theorems \ref{thm:coHom2Main:intro} and \ref{thm:CSAMain:intro} combine to prove the following faithful generalisation of their result.
(This corollary comes before Theorem \ref{thm:coHom2Main:intro} in the paper, as it does not need the full statement of Theorem \ref{thm:coHom2Main:intro}).

\begin{corolletter}[Corollary \ref{corol:codimCSA}]
\label{corol:codimCSA:intro}
$\BS$-free groups of cohomological dimension $2$ are CSA.
\end{corolletter}

\p{Balanced groups}
A group $G$ is \emph{balanced} if for every non-trivial $x\in G$, if the elements $x^m$ and $x^n$ are conjugate then $|m|=|n|$ (the terminology was introduced by Wise \cite{Wise2000graphs}).
Therefore, a group is weakly AH if it is balanced and contains no $\mathbb{Z}^2$-subgroups.
There are certain connections between the study of weakly AH groups presented here and the theory of balanced groups.
In particular, in an unpublished article Button also proves that the ``Gildenhuys groups'' in Theorem \ref{thm:BSvsGvsA}.(\ref{BSvsGvsA:1}) are $\BS$-free and (implicitly) commutative-transitive, but not weakly AH \cite[Example following Definition 3.4]{button2015balanced}.
Additionally, Proposition \ref{prop:residualWAH} on residual properties is similar to Proposition 2.4 of Button's article, with the difference being that we stipulate no $\mathbb{Z}^2$ subgroups in both our statement and conclusion.
Button's arguments here are rather different from ours, as they are based on an embedding of these groups into $\operatorname{SL}_2(\mathbb{C})$, which suggests an alternative path through Section \ref{sec:Characterisations}.

However, having no $\mathbb{Z}^2$ subgroups is integral to this paper.
In particular, non-positively curved variants of Questions \ref{Question:1} and \ref{Question:2}, where one asks if (low dimensional) balanced groups of type $\mathtt F$ are necessarily CAT(0) or automatic, have a negative answer, even when we restrict to the original context of Gersten's question, namely one-relator groups \cite{GardamWoodhouse19}.

\p{Outline of the paper}
In Section \ref{sec:definition} we introduce torsion-free algebraically hyperbolic (AH) groups, and we show that $\BS$-free groups of cohomological dimension $2$ are AH.
In Section \ref{sec:CSA} we prove Theorem \ref{thm:CSAMain:intro}, that AH groups are CSA.
In Section \ref{sec:WeaklyAHG} we prove the inclusion $\text{AH}\subseteq\text{weakly AH}\subseteq\text{$\BS$-free}$, and we prove Theorem \ref{thm:coHom2Main:intro}.
In Section \ref{sec:AHvsWAH} we investigate the inclusion $\text{AH}\subseteq\text{weakly AH}$, linking this to cyclic extensions of infinite torsion groups (``Adian extensions'').
In particular, we prove that this inclusion is proper, i.e.\ $\text{AH}\subset\text{weakly AH}$, we prove Theorem \ref{thm:LI:intro}, and in Corollary \ref{corol:AnswerQn} we justify the importance of Question~\ref{qn:WAHvsAH}.
In Section \ref{sec:Characterisations} we give a characterisation of AH groups in terms of ``poison'' subgroups, i.e.\ in terms of groups which can never be subgroups of AH groups (Theorem \ref{thm:SecondClassification}) and we prove Theorem \ref{thm:inclusions:intro}.
In Section \ref{sec:Residual} we consider residually weakly AH groups and fully residually AH groups.
Finally, Section \ref{sec:splittings} builds on Theorem \ref{thm:CSAMain:intro} and the CSA-property:
As AH groups are CSA, they have a JSJ-theory, and Section \ref{sec:splittings} characterises the abelian splittings (hence, the abelian JSJ decompositions) of two-generated AH groups.
The characterisation in Section \ref{sec:splittings} is entirely analogous to the situation in two-generated torsion-free hyperbolic groups, but additionally sometimes allows for non-cyclic splittings.

\p{Acknowledgments}
We are grateful to
Ashot Minasyan for multiple helpful discussions and telling us about a result of Adian which allowed us to prove the properness of the inclusion $\text{AH}\subset\text{weakly AH}$,
to Robert Kropholler for helpful correspondence on finiteness properties,
and to Jack Button for drawing the helpful parallels between our work and the theory of balanced groups.
We thank the referee for their comments that improved the paper.

This work has received funding from
the European Research Council (ERC) under the European Union's Horizon 2020 research and innovation programme (Grant agreement No. 850930),
the European Union (ERC, SATURN, 101076148),
the Deutsche Forschungsgemeinschaft (DFG, German Research Foundation) -- Project-ID 427320536 -- SFB 1442, as well as under Germany's Excellence Strategy EXC 2044--390685587 and EXC-2047/1 -- 390685813,
and from the Engineering and Physical Sciences Research Council (EPSRC), grants EP/R035814/1 and EP/S010963/1.


\section{Algebraically hyperbolic groups}
\label{sec:definition}
In this section we study basic properties and examples of algebraically hyperbolic groups, and we prove that $\BS$-free groups of cohomological dimension $2$ are algebraically hyperbolic.

\subsection{Main definition}
\label{sec:lccConditions}
All the groups studied in this article are torsion-free, so we often do not state the torsion-free condition for the sake of brevity.
Recall from \cref{def:ah} that a torsion-free group is AH if it \ref{AHdef:CT} is commutative-transitive and \ref{AHdef:BS} contains no subgroup $H \rtimes \mathbb{Z}$ with $H$ non-trivial locally cyclic.

\begin{rmk}
	Subgroups of AH groups are AH, and subgroups of weakly AH groups are weakly AH.
\end{rmk}

Condition \ref{AHdef:CT} is equivalent to centralisers of non-trivial elements being abelian.
As Condition \ref{AHdef:BS} rules out $\mathbb{Z}^2$ subgroups, centralisers -- and indeed all abelian subgroups -- are in fact locally cyclic.
Therefore, an AH group is equivalently
one satisfying Condition \ref{AHdef:BS} and:
\begin{enumerate}[label=(\alph*)]
\setcounter{enumi}{\value{nameOfYourChoice}}
\item\label{AHdef:cyclic} every non-trivial element of $G$ has locally cyclic centraliser.
\end{enumerate}
Note that by \emph{locally cyclic} we mean that non-trivial finitely generated subgroups are isomorphic to $\mathbb{Z}$, so Condition \ref{AHdef:cyclic} includes torsion-freeness.
A locally cyclic group is equivalently a subgroup of $\mathbb{Q}$ (after sending some non-trivial $g \mapsto 1 \in \mathbb{Q}$ there is a unique extension to an injective homomorphism $G \to \mathbb{Q}$).
We usually work with Conditions \ref{AHdef:CT} and \ref{AHdef:BS}.
However, sometimes it is convenient to work with Conditions \ref{AHdef:BS} and \ref{AHdef:cyclic}, as in the following lemma.

\begin{lemma}
\label{lem:cyclicCentralisers}
Let $G$ be a group such that every non-trivial element has infinite cyclic centraliser. Then $G$ is AH.
\end{lemma}

\begin{proof}
Condition \ref{AHdef:cyclic} is immediate.
For Condition \ref{AHdef:BS},
any $H\rtimes\mathbb{Z}$ subgroup, with $H$ non-trivial locally cyclic, is of the form $\mathbb{Z}\rtimes\mathbb{Z}$, and hence is either $\mathbb{Z}^2$ or contains this group with index $2$.
This is impossible as centralisers are cyclic, and the result follows.
\end{proof}

\subsection{First examples}

\begin{example}
Torsion-free subgroups of hyperbolic groups are AH. As centralisers of elements are infinite cyclic \cite[Corollary III.$\Gamma$.3.10]{BH1999Metric}, this follows from Lemma \ref{lem:cyclicCentralisers}.
\end{example}

Later, in \Cref{ex:residualAH}, we generalise the above to show that limit groups of torsion-free hyperbolic groups which contain no $\mathbb{Z}^2$ subgroups are AH.

\begin{example}
A torsion-free Tarski monster, i.e. a non-cyclic group whose proper non-trivial subgroups are all isomorphic to $\mathbb{Z}$ as first constructed by Olshanskii \cite{Olshanskii79noetherian}, is AH.
\end{example}

\begin{lemma}
	Free products of AH groups are AH.
\end{lemma}
\begin{proof}
	Let $G$ be a free product of AH groups $G_i$. Then $G$ acts on the associated Bass--Serre tree $T$ (we will review the theory of groups acting on trees more fully in \cref{gogs}).
	
	
	Now let $g,h,k \in  G$ be non-trivial elements such that $h$ commutes with $g$ and $k$. If the fixed-point set of $h$ is empty, then $h$ acts hyperbolically by pushing along an axis. This axis has to be preserved by $g$ and $k$, and using the fact that the action of $G$ is free on the edges of $T$, we easily see that $g$ and $k$ have positive powers that coincide.
	
	If the fixed-point set of $h$ is non-empty, then it is a single point, and therefore the elements $g,h,$ and $k$ all lie in the same conjugate of some group $G_i$. This forces $g$ to $k$ to commute.
	
	If $G$ contains $\mathbb Z^2$, then let $g$ and $h$ be generators of such a subgroup, and let $k = gh$. Then the above argument shows that all three elements must fix a vertex, and hence they all lie in a conjugate of some $G_i$, which is impossible.
\end{proof}

\begin{figure}[h]
  \centering
  \includegraphics[width=0.7\textwidth]{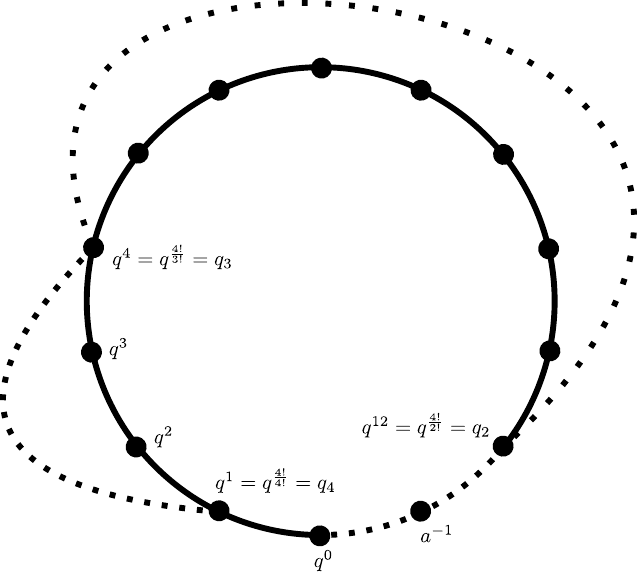}
  \caption{The graph $\Gamma_4$}
  \label{graph Gamma}
\end{figure}

\begin{example}
The free product $\mathbb{Q} * \mathbb{Q}$ is AH.
There are also finitely generated examples of AH groups that contain a non-cyclic locally cyclic subgroup; let us construct one.

Consider $\mathbb Q = \langle \{ q_i : i \geqslant 1\}  \mid {q_{i+1}}^{i+1} = q_i, i \geqslant 1 \rangle$ and let $F_2 = F(a,b)$. Let $H = \mathbb Q * F_2$, and consider two subgroups of $H$, namely
\[
A = \langle \{ q_i a^i : i \geqslant 2\}\rangle \quad \textrm{and} \quad B = \langle \{ q_i b^i : i \geqslant 3\}\rangle.
\]
Observe that both groups are freely generated by the indicated generating sets.
Let $G$ be the HNN extension of $H$ obtained by identifying $A$ and $B$ by sending $q_i a^i \mapsto q_{i+1} b^{i+1}$. The group $G$ is generated by $a,b,q_2$, and the stable letter of the extension.

We claim that $G$ is AH. To prove this, we first need to argue that $A$ is malnormal in $H$.
If it is not, then there exist $x \in A \smallsetminus \{1\}$ and $y \in H \smallsetminus A$ such that $y^{-1}xy \in A$. Clearly, there exists $n$ such that
\[\{ x,y^{-1}xy \}  \subset A_n=\langle \{ q_i a^i : n\geqslant i \geqslant 2\}\rangle\]
and $y \in \langle q_n\rangle *F_2$. Hence, it suffices to prove that $A_n$ is malnormal
inside of $\langle q_n \rangle * F_2$, for every $n$.

 Let $n$ be fixed, and set $q = q_n$.
Stallings folding shows that $A_n$ is the fundamental group of the graph $\Gamma_n$ with vertex set 
\[\left\{q^j \colon 0 \leqslant j \leqslant \frac {n!} 2 \right\} \sqcup \{a^{-1}\},\]
 an edge labelled $q$ connecting $q^j$ to $q^{j+1}$ for every $0\leqslant j<\frac{n!} 2$,
 an edge labelled $a$ connecting $q^{\frac {n!}{j!}}$ to $q^{\frac {n!}{(j-1)!}}$ for $2<j\leqslant n$,
 	and two edges labelled $a$ connecting $q^{\frac {n!}{2}}$ to $a^{-1}$ and  $a^{-1}$ to $q^0$ (see \cref{graph Gamma}).

To prove malnormality we need to consider a non-trivial loop $\lambda$ in $\Gamma_n$ without backtracking. The loop yields a cyclically reduced word in the alphabet $\{q,a,b\}$; consider a maximal subword of the form $q^k$, and note that we have $k>0$ without loss of generality (inverting the loop if needed). The structure of the graph $\Gamma_n$ tells us that either $k = \frac{n!}{\beta!} - \frac{n!}{\alpha!}$ for some integers $n\geqslant \alpha>\beta \geqslant 2$, or $k = \frac{n!}{\beta!}$ with $n\geqslant \beta \geqslant 2$, depending on whether the subpath of the loop $\lambda$ corresponding to $q^k$ starts at $q^{\frac{n!}{\alpha!}}$ or at $q^0$; in both cases it ends at $q^{\frac{n!}{\beta!}}$.
In either case, the fact that $\beta \geqslant 2$ implies that we have $\frac{n!}{(\beta+1)!}<k \leqslant \frac{n!}{\beta!}$. Hence $k$ determines $\beta$ uniquely. This implies that if 
the loop can be traced in $\Gamma_n$, then the subpath corresponding to $q^k$ must end at $q^{\frac{n!}{\beta!}}$, and so $\lambda$ can only be traced inside of $\Gamma_n$ in one way.
We conclude that $A_n$ is malnormal in $\langle q_n \rangle *F_2$, as claimed.

An analogous proof shows that $B$ is malnormal, and it is extremely easy to see using the above graph that the conjugates of $A$ and $B$ in $H$ intersect trivially. The upshot is that $G$ acts on the Bass--Serre tree corresponding to the HNN structure in such a way that no non-trivial element of $G$ fixes more than a single edge.

The rest of the proof follows the same outline as in the previous lemma: using its notation, if the element $h$ acts hyperbolically, then a product of non-trivial powers of $g$ and $k$ fixes an axis pointwise, and hence must be trivial. If $h$ fixes a vertex, then its fixed-point set is either a single vertex or two vertices joined by an edge. Since the action of $G$ does not invert edges, both $g$ and $k$ fix the fixed-point set of $h$ pointwise.
\end{example}

\subsection{Adian-free groups}
In our proofs that specific classes of groups are AH, we prove commutative-transitivity by proving that the groups in question are \emph{Adian-free}, which is a re-packaging of commutative-transitivity for $\mathbb{Z}^2$-free groups in terms of poison subgroups.
We now give this re-packaging before using it to obtain further examples of AH groups.

Our motivation for this re-packaging is primarily the neater proofs it yields, and in terms of statements of results it is purely cosmetic as we are always able to substitute ``Adian-free'' for the more familiar ``commutative-transitive'' property (see for example Theorem \ref{thm:wAHclassification} and Corollary \ref{corol:FirstClassification}).
The cosmetics of the change is beneficial though, as for example it lets us characterise AH groups in terms of poison subgroups (Theorem \ref{thm:SecondClassification}).

\begin{definition}
\label{def:adian}
An \emph{Adian extension} is a finitely generated torsion-free group which maps onto an infinite torsion group with cyclic kernel.
\end{definition}
We name these groups for Adian, who constructed the first examples of such groups \cite{Adian1971Certain}.
(Note that ``Adian groups'' are different \cite{Stallings1987Adian}.)
We call a group \emph{Adian-free} if it contains no Adian extension as a subgroup.

\begin{lemma}
\label{lem:CTvsAdian}
Suppose a torsion-free group $G$ has no $\mathbb{Z}^2$-subgroups. Then $G$ is commutative-transitive if and only if it is Adian-free.
\end{lemma}

\begin{proof}
Suppose $G$ is not Adian-free, so that it contains an Adian extension, i.e.\ a finitely generated subgroup $K$ with a normal, cyclic subgroup $Z$ such that $K/Z$ is an infinite torsion group.
As $K$ is torsion-free, for every $k\in K\smallsetminus\{1\}$ there exists some $i\in\mathbb{Z}$ such that $k^i\in Z\smallsetminus\{1\}$.
Therefore, every pair of non-trivial elements $k_1, k_2\in K\smallsetminus\{1\}$ has a common, non-trivial power.
This means that for all $k_1, k_2\in K\smallsetminus\{1\}$ there exists $z\in Z \smallsetminus\{1\}$ such that $[k_1, z]=1=[z, k_2]$.
As $K$ is not abelian, this implies that $G$ is not commutative-transitive.

Suppose $G$ is Adian-free and let $k_1, k_2, z\in G\smallsetminus\{1\}$ be arbitrary elements such that $[k_1, z]=1=[z, k_2]$ but $[k_1, k_2]\neq1$.
Write $Z:=\langle z\rangle$ and $K:=\langle k_1, k_2, z\rangle$.
As $G$ contains no $\mathbb{Z}^2$ subgroups,
and as $Z$ is central in $K$, the subgroups $\langle k, z\rangle$ are cyclic for all $k\in K$.
Therefore, for all $k\in K\smallsetminus\{1\}$ there exists some $i\in\mathbb{Z}\smallsetminus\{0\}$ such that $k^i\in Z$, and so $K/Z$ is a torsion group.
As $G$ is Adian-free, $K/Z$ is finite, and thus $K$ is torsion-free and virtually cyclic, and hence cyclic, and so $[k_1, k_2]=1$.
Thus $G$ is commutative-transitive.
\end{proof}

\subsection{{\boldmath$\BS$}-free groups of cohomological dimension {\boldmath$2$}}
\label{sec:coHomBS}
Our next class of examples of AH groups is the class of $\BS$-free groups of cohomological dimension $2$, which includes for example all $\BS$-free torsion-free one-relator groups.

We recall the following relatively well-known fact.
\begin{lemma}
\label{lem:bs_contains_metabelian_bs}
Every Baumslag--Solitar group contains a $\BS(1, k)$ subgroup for some $|k|\geqslant1$.
\end{lemma}
\begin{proof}
Considering the infinite cyclic cover of the graph of $\mathbb{Z}$s for \[\BS(m, n) = \langle a, t\mid t^{-1}a^mt=a^n\rangle\] with $|m|,|n|>1$ we see that $a$ and $a^t$ generate a group isomorphic to $\langle x, y \mid x^n = y^m \rangle$.
The quotient of the latter by its central $\mathbb{Z} = \langle x^n \rangle = \langle y^m \rangle$ is $\mathbb{Z} / n * \mathbb{Z} / m$ and the preimage of any infinite cyclic subgroup is thus isomorphic to $\mathbb{Z}^2$.
For example, \[ \langle xy, x^n \rangle \cong \langle a t^{-1} a t, a^n \rangle \cong \mathbb{Z}^2 \cong \BS(1, 1).\]

For $\BS(1, k)$ with $k < 0$ we have $\langle a, t^2 \rangle \cong \BS(1, k^2)$.
\end{proof}

\begin{proposition}
\label{prop:coHomBS}
Let $G$ be a group of cohomological dimension $2$. Then $G$ is $\BS$-free if and only if $G$ is AH.
\end{proposition}

\begin{proof}
If $G$ is AH then it contains no subgroups of the form $H\rtimes\mathbb{Z}$ for $H$ non-trivial locally cyclic.
As every Baumslag--Solitar group contains a $\BS(1, k)$ subgroup for some $|k|\geqslant1$, and as $\BS(1, k) \cong \mathbb{Z}[\frac{1}{k}] \rtimes \mathbb{Z}$ has the form $H\rtimes\mathbb{Z}$ for $H$ locally cyclic, $G$ is $\BS$-free as required.

Now suppose that $G$ is $\BS$-free.
Then $G$ contains no $H\rtimes\mathbb{Z}$ subgroup with $H$ locally cyclic, and to see this suppose otherwise.
As $G$ has cohomological dimension $2$, the subgroup $H\rtimes\mathbb{Z}$ has cohomological dimension at most $2$, and is soluble.
Hence, $H\rtimes\mathbb{Z}\cong\BS(1, k)$ for some $|k|\geqslant1$ \cite[Theorem 5]{Gildenhuys1979Classification}, a contradiction as $G$ is $\BS$-free.

For commutative-transitivity, we first prove that Adian extensions have cohomological dimension at least $3$.
So, let $K$ be a finitely generated group of cohomological dimension at most $2$ and with a normal, cyclic subgroup $Z$ such that $K/Z$ is a torsion group.
In general, the quotient of a finitely generated group of cohomological dimension at most $2$ by a finitely generated normal free subgroup is free-by-finite \cite[Corollary 8.7]{Bieri1981Homological}, and so in our special case the torsion quotient $K/Z$ must in fact be finite.
Hence, $K$ is torsion-free and virtually $\mathbb{Z}$, and so cyclic.
In particular, $K$ is not an Adian extension, and so Adian extensions have cohomological dimension at least $3$.
Therefore, as subgroups of $G$ have cohomological dimension at most $2$, we have that $G$ is Adian-free, and hence is commutative-transitive by Lemma \ref{lem:CTvsAdian}.
The result follows.
\end{proof}

We are not yet ready to prove Theorem \ref{thm:coHom2Main:intro}, as we additionally need the inclusions $\text{AH}\subseteq\text{weakly AH}\subseteq\text{$\BS$-free}$. Our proof of the first of these requires the fact that AH groups are CSA.


\section{The CSA property}
\label{sec:CSA}
Recall from the introduction that a group $G$ is \emph{CSA} if all maximal abelian subgroups of $G$ are malnormal in $G$.
In this section we prove that torsion-free AH groups are CSA.
This is a useful result as it promotes commutative-transitivity to the stronger CSA property.
We later apply it to prove that AH groups are weakly AH (Lemma \ref{lem:AHimpliesWAH}).
Moreover, it means that AH groups have a JSJ-theory.

We begin with two preliminary results on commutative-transitive groups.
The first also characterizes commutative-transitive groups.

\begin{lemma}
\label{lem:maxab_is_centraliser}
Let $G$ be commutative-transitive group. Then
\begin{enumerate}[label=(\roman*)]
\item every maximal abelian subgroup of $G$ is the centraliser of each of its non-trivial elements, and
\item the centraliser of every non-trivial element or non-trivial abelian subgroup is maximal abelian.
\end{enumerate}
\end{lemma}

\begin{proof}
To show (i), let $A \leqslant G$ be maximal abelian and let $a \in A \smallsetminus \{1\}$.
Then $A \leqslant C_G(a)$ and by assumption the latter is abelian.
Thus by maximality we have $A = C_G(a)$.

For (ii), if $a \in G$ is non-trivial then $C_G(a) = C_G(\langle a \rangle)$ so it suffices to consider a non-trivial abelian subgroup $A \leqslant G$.
By commutative transitivity, $C_G(A)$ is abelian, and by definition it contains every abelian subgroup containing $A$ so in particular it contains every abelian subgroup containing itself, so it is maximal abelian.
\end{proof}

\begin{lemma}
\label{lem:rootclosed}
Let $G$ be torsion-free commutative-transitive and let $A = C_G(a)$ be the centraliser of a non-trivial element $a \in G$.
Then $A$ is root-closed, that is, if $g \in G$ satisfies $g^k \in A$ for $k \neq 0$ then $g \in A$.
\end{lemma}

\begin{proof}
By \cref{lem:maxab_is_centraliser} we have $A = C_G(g^k)$ but $[g, g^k] = 1$ so $g \in A$.
\end{proof}

The main result of this section follows quickly from the following classification of AH groups.

\begin{proposition}
\label{prop:ah_via_centralisers}
A group $G$ is AH if and only if centralisers of non-trivial elements of $G$ are locally cyclic and self-normalizing.
\end{proposition}

\begin{proof}
Suppose first that $G$ is AH.
As noted with Condition \ref{AHdef:cyclic} in \cref{sec:lccConditions}, every non-trivial element has locally cyclic centraliser.
Let $h \in G$ be non-trivial and let $H = C_G(h)$.
Suppose that $g \in N_G(H)$, that is, $H^g = H$.
If $g \notin H$ then $\langle H, g \rangle / H \cong \mathbb{Z}$ as $H$ is root-closed (\cref{lem:rootclosed}) so $\langle H, g \rangle \cong H \rtimes \mathbb{Z}$, contradicting algebraic hyperbolicity.
Thus $N_G(H) = H$, that is, $H$ is self-normalizing.

Suppose now that centralisers of non-trivial elements of $G$ are locally cyclic and self-normalizing.
Then Condition \ref{AHdef:cyclic} is immediate.
It remains to show Condition \ref{AHdef:BS} holds, so suppose for the sake of contradiction that $G$ contains a subgroup $H \rtimes \mathbb{Z}$ with $H$ non-trivial locally cyclic and $\mathbb{Z} = \langle t \rangle$.
Now write $\overline{H} = C_G(H)$. 
Then since $\overline{H}^t = C_G(H^t) = C_G(H) = \overline{H}$ and $\overline{H}$ is self-normalizing, we have $t \in \overline{H}$.
But $\overline{H}$ is locally cyclic, which is impossible as $\langle t \rangle \cap H = 1$.
Thus $G$ is AH.
\end{proof}

\begin{theorem}[Theorem \ref{thm:CSAMain:intro}]
\label{thm:CSA}
Torsion-free AH groups are CSA.
\end{theorem}

\begin{proof}
Let $G$ be AH and let $A \leqslant G$ be a maximal abelian subgroup.
Suppose that $g \in G$ satisfies $A^g \cap A \neq 1$ and let $h \in A^g \cap A$ be non-trivial.
As conjugation by $g$ is an automorphism of $G$, $A^g$ is also a maximal abelian subgroup and so by \cref{lem:maxab_is_centraliser} we have $A^g = C_G(h) = A$.
Now by \cref{prop:ah_via_centralisers} $A$ is self-normalizing, so $g \in A$.
Thus $A$ is malnormal.
\end{proof}

\p{{\boldmath$\BS$}-free groups of cohomological dimension {\boldmath $2$}}
We can now prove Corollary \ref{corol:codimCSA:intro}.

\begin{corollary}[Corollary \ref{corol:codimCSA:intro}]
\label{corol:codimCSA}
$\BS$-free groups of cohomological dimension $2$ are CSA.
\end{corollary}

\begin{proof}
The result is immediate by combining Proposition \ref{prop:coHomBS} and Theorem \ref{thm:CSA}.
\end{proof}

\p{Cyclic centralisers}
Theorem \ref{thm:CSA} gives a new proof of the following result of Myasnikov and Remeslennikov \cite[Proposition 11]{myasnikov1996exponential}.

\begin{corollary}[Myasnikov--Remeslennikov]
Let $G$ be a group such that every non-trivial element has infinite cyclic centraliser. Then $G$ is CSA.
\end{corollary}

\begin{proof}
The group $G$ is AH by Lemma \ref{lem:cyclicCentralisers}, so CSA by Theorem \ref{thm:CSA}.
\end{proof}


\section{Relationships between the three definitions}
\label{sec:WeaklyAHG}
In this section we consider the connections between AH, weakly AH, and $\BS$-free groups.
%
We start by proving the inclusions $\text{AH}\subseteq\text{weakly AH}\subseteq\text{$\BS$-free}$ (the second of these is well known).

\begin{lemma}
\label{lem:AHimpliesWAH}
Let $G$ be a torsion-free group.
\begin{enumerate}
\item\label{AHimpliesWAH:1} If $G$ is AH then $G$ is weakly AH.
\item\label{AHimpliesWAH:2} If $G$ is weakly AH then $G$ is $\BS$-free.
\end{enumerate}
\end{lemma}

\begin{proof}
For (\ref{AHimpliesWAH:1}), let $G$ be an AH group and note that $G$ has no $\mathbb{Z}^2$ subgroups.
Suppose $x\in G\smallsetminus\{1\}$ is such that $x^m$ and $x^n$ are conjugate, so $g^{-1}x^mg=x^n$.
Let $A$ be a maximal abelian subgroup of $G$ containing $x$.
Then $g^{-1}Ag \cap A$ is non-trivial, and therefore, by Theorem \ref{thm:CSA}, $g\in A$.
In particular, $g$ commutes with $x$ and so $m=n$.
Hence, $G$ is weakly AH.

For (\ref{AHimpliesWAH:2}), suppose that $G$ contains a Baumslag--Solitar subgroup.
Then $G$ contains a $\BS(1, k)$-subgroup for $|k|\geqslant1$ by  \cref{lem:bs_contains_metabelian_bs}.
It follows that $G$ is not weakly AH, as either we have a $\mathbb{Z}^2$-subgroup (if $|k|=1$) or an element conjugate to a proper power (if $|k|>1$).
The result follows from the contrapositive.
\end{proof}

We now prove Theorem \ref{thm:coHom2Main:intro} from the introduction, which is as follows.

\begin{theorem}[Theorem \ref{thm:coHom2Main:intro}]
\label{thm:coHom2Main}
Let $G$ be a group of cohomological dimension $2$. Then the following are equivalent.
\begin{enumerate}
\item $G$ is AH.
\item $G$ is weakly AH.
\item $G$ is $\BS$-free.
\end{enumerate}
\end{theorem}

\begin{proof}
Let $G$ have cohomological dimension $2$.
By Proposition \ref{prop:coHomBS}, if $G$ is $\BS$-free then $G$ is AH.
The result then follows by Lemma \ref{lem:AHimpliesWAH}.
\end{proof}


\section{AH versus weakly AH groups}
\label{sec:AHvsWAH}
In this section we focus on the inclusion ``$\text{AH}\subseteq\text{weakly AH}$'' from Lemma \ref{lem:AHimpliesWAH}.
In particular, we prove that this inclusion is proper (Theorem \ref{thm:weaklyAHvsAH}), and we prove Theorem \ref{thm:LI:intro}, that the classes are equal in the case of both locally indicable groups and residually finite groups.

The results of this section are centered around a characterisation of those weakly AH groups which are not AH (Theorem \ref{thm:wAHclassification}).

\p{Adian extensions}
We begin by returning to Adian extensions, as defined in \cref{def:adian}, which give a poison subgroup characterisation of commutative-transitivity.
We now prove that these groups sit in the gap between AH and weakly AH groups, and in Theorem \ref{thm:wAHclassification}, below, we see that these groups characterise this gap.

\begin{lemma}
\label{lem:TorsionGroups}
Adian extensions are weakly AH but not AH.
\end{lemma}

\begin{proof}
Let $K$ be an Adian extension, and let $Z$ be the normal, cyclic subgroup of $K$ such that $K/Z$ is an infinite torsion group.
Then every pair of non-trivial elements $k_1, k_2\in K\smallsetminus\{1\}$ has a common, non-trivial power, so $K$ contains no $\mathbb{Z}^2$-subgroup, while if $g^{-1}x^mg=x^n$ with $x$ non-trivial then taking $i\in\mathbb{Z}\smallsetminus \{0\}$ such that $x^i$ is a power of $g$ we have $x^i=g^{-1}x^ig$, which gives
\[
x^{im}=(g^{-1}x^ig)^m=(g^{-1}x^mg)^i=(x^n)^i=x^{in}
\]
and so $im=in$, so $m=n$.
Therefore, $K$ is weakly AH.

Then $K$ is not AH because, as observed above, it contains no $\mathbb{Z}^2$ subgroups, and so by Lemma \ref{lem:CTvsAdian} it is not commutative-transitive.
\end{proof}

As Adian extensions do indeed exist, there exist weakly AH groups which are not AH.
Theorem \ref{thm:weaklyAHvsAH}, below, gives examples with specific properties.

We now give our characterisation of those weakly AH groups which are not AH.

\begin{theorem}
\label{thm:wAHclassification}
A torsion-free group $G$ is AH if and only if
\begin{enumerate}
\item\label{wAHclassification:1} $G$ is Adian-free, and
\item\label{wAHclassification:2} $G$ is weakly AH.
\end{enumerate}
\end{theorem}

\begin{proof}
Suppose $G$ is AH.
Then (\ref{wAHclassification:2}) holds, by Lemma \ref{lem:AHimpliesWAH}, while all subgroups of $G$ are also AH, and so (\ref{wAHclassification:1}) holds by Lemma \ref{lem:TorsionGroups}.

Suppose that $G$ is weakly AH and Adian-free.
Then $G$ has no $\mathbb{Z}^2$ subgroups, so $G$ is commutative-transitive by Lemma \ref{lem:CTvsAdian}.
To see that $G$ contains no $H\rtimes\mathbb{Z}$ subgroup with $H$ non-trivial locally cyclic, suppose otherwise.
Then, for such a subgroup $H\rtimes\mathbb{Z}$ with the $\mathbb{Z}$-factor generated by an element $t$, there exists $x\in H$ and $m, n\in\mathbb{Z}\smallsetminus\{0\}$ such that $t^{-1}x^mt=x^n$. Then $|m|=|n|$ and so $\langle t^2, x^m\rangle\cong\mathbb{Z}^2$, contradicting weak AH.
Therefore, $G$ is commutative-transitive and contains no such $H\rtimes\mathbb{Z}$ subgroups, and so is AH as required.
\end{proof}

As Lemma \ref{lem:CTvsAdian} allows us to swap between ``commutative-transitive'' and ``Adian-free'', we have the following.

\begin{corollary}
\label{corol:FirstClassification}
A torsion-free group $G$ is AH if and only if
\begin{enumerate}
\item\label{FirstClassification:1} $G$ is commutative-transitive, and
\item\label{FirstClassification:2} $G$ is weakly AH.
\end{enumerate}
\end{corollary}

\begin{proof}
As both AH and weakly AH groups contain no $\mathbb{Z}^2$-subgroups, Lemma \ref{lem:CTvsAdian} is applicable and so commutative-transitivity is equivalent to being Adian-free. The result now follows from Theorem \ref{thm:wAHclassification}.
\end{proof}

\p{Separating AH and weakly AH}
Lemma \ref{lem:TorsionGroups} and Theorem \ref{thm:wAHclassification} combine with some deep results to give the following:

\begin{theorem}
\label{thm:weaklyAHvsAH}
There exist weakly AH groups $G$ which are not AH.
Moreover, the groups $G$ can be taken to be finitely generated or simple.
\end{theorem}

\begin{proof}
Adian proved the existence of what we call Adian extensions, that is, there exist finitely generated, torsion-free groups that are cyclic extensions of infinite torsion groups \cite{Adian1971Certain}.
Therefore, by Lemma \ref{lem:TorsionGroups} there exist finitely generated weakly AH groups which are not AH.

Obraztsov gave a refined version of Adian's theorem, by proving that there exist infinite simple groups such that their every finitely generated subgroup is either cyclic or a two-generated Adian extension \cite[Theorem A]{Obraztsov1998Simple}.
Now, locally (weakly AH) implies weakly AH, and so Obraztsov's groups are weakly AH by Lemma \ref{lem:TorsionGroups}.
On the other hand, by simplicity they are not locally cyclic, and so contain Adian extensions, and so by Theorem \ref{thm:wAHclassification} Obraztsov's groups are not AH, as required.
\end{proof}

\p{Equating AH and weakly AH}
We now prove Theorem \ref{thm:LI:intro}.
Recall from the introduction that a group $G$ is \emph{locally indicable} if every non-trivial finitely generated subgroup $K$ of $G$ surjects onto the infinite cyclic group, $K\twoheadrightarrow \mathbb{Z}$, and is \emph{residually finite} if the intersection of all of its finite-index subgroups is trivial.

\begin{theorem}[Theorem \ref{thm:LI:intro}]
\label{thm:LI}
Let $G$ be a locally indicable group or a torsion-free residually finite group.
Then $G$ is weakly AH if and only if $G$ is AH.
\end{theorem}

\begin{proof}
Locally indicable groups are torsion-free, and so by Theorem \ref{thm:wAHclassification} it is sufficient to prove that locally indicable groups and residually finite groups are Adian-free.
This is known for residually finite groups \cite{Sozutov2000Residually}, so we need only consider locally indicable groups.

Suppose $G$ is locally indicable or, more generally, that every non-trivial finitely generated subgroup of $G$ is either infinite cyclic or has a proper non-torsion quotient.
Then $G$ is Adian-free by the following proposition.
\end{proof}

\begin{proposition}
Every proper quotient of an Adian extension is torsion.
\end{proposition}

In analogy to just infinite groups, we could call this property being \emph{just non-torsion}.

\begin{proof}
Suppose that $K$ is an Adian extension, so $K$ is finitely generated torsion-free and there is an infinite cyclic subgroup $Z = \langle z \rangle \lhd K$ with $K / Z$ infinite torsion.
If $1 \neq N \lhd K$ then we can pick some non-trivial $n \in N$ and since $K / Z$ is torsion, there exists an integer $r \geqslant 1$ such that $n^r \in Z$, say $n^r = z^s$ where $s \in \mathbb{Z} \smallsetminus \{0\}$ (by torsion-freeness).
Then the quotient $K \to K / N$ factors through $K / \langle z^s \rangle$, which is an extension of $\langle z \rangle / \langle z^s \rangle$ by $K / Z$ and thus, being torsion-by-torsion, is itself torsion.
Hence $K / N$ is torsion.
\end{proof}

\p{Groups of type \boldmath{$\mathtt F$}}
We end this section by discussing Question~\ref{qn:WAHvsAH} from the introduction.
Let us rephrase this question in terms of Adian extensions:
\begin{question}[Question~\ref{qn:WAHvsAH}]
Do Adian extensions have infinite cohomological dimension?
\end{question}
The statement above is equivalent to that of \cref{qn:WAHvsAH}, since groups with non-trivial torsion always have infinite cohomological dimension.

If the above question has a negative answer, then every counterexample divided by its centre gives a finitely generated  infinite torsion group whose cohomology is periodic of period $2$ after finitely many steps. Groups with this property have been investigated in the literature, for example by Talelli~\cite{Talelli2012}.

\smallskip

The following corollary of Theorem \ref{thm:wAHclassification} explains the importance of \cref{qn:WAHvsAH}.

\begin{corollary}
\label{corol:AnswerQn}
There exists an Adian extension that embeds into a weakly AH group of type $\mathtt F$ if and only if there exists a weakly AH group of type $\mathtt F$ which is not AH.
In particular, if Adian extensions have infinite cohomological dimension then a group $G$ of type $\mathtt F$ is weakly AH if and only if it is AH.
\end{corollary}

\begin{proof}
The first sentence is immediate from Theorem \ref{thm:wAHclassification}, while the second sentence holds as subgroups of groups $G$ of type $\mathtt{F}$ have finite cohomological dimension.
\end{proof}


\section{Poison subgroups}
\label{sec:Characterisations}
In this section we investigate certain poison subgroups, that is certain finitely generated groups which can never occur as subgroups of AH groups.
In particular, Theorem \ref{thm:SecondClassification} is a poison-subgroup characterisation of AH groups.
This is similar in spirit to the definition of $\BS$-free groups, but we must exclude two additional classes of groups as subgroups: Adian extensions, and ``{Gildenhuys groups}''.
From here, we prove the existence of $\BS$-free groups which are not weakly AH, which completes the proof of Theorem \ref{thm:inclusions:intro}.

\p{Gildenhuys' groups, and the groups {\boldmath$G_{m, n}$}}
Recall the notation $\mathbb{Z}\left[1/mn\right]$, with $m, n\in\mathbb{Z}\smallsetminus\{0\}$, for the subgroup of $(\mathbb{Q}, +)$ consisting of fractions of the form $\frac{x}{m^pn^q}$, where $x, p, q\in\mathbb{Z}$.
\begin{definition}
Let $m$ and $n$ be non-zero coprime integers.
We define the group \[G_{m, n} := \mathbb{Z}\left[1/mn\right]\rtimes\mathbb{Z} \] where the action of the generator of the $\mathbb{Z}$-factor is multiplication by $m/n$.
If $|m|, |n|>1$ then we call $G_{m, n}$ a \emph{Gildenhuys group}.
\end{definition}
We name these groups for Gildenhuys, who proved that they have cohomological dimension $3$ \cite[Theorem 4]{Gildenhuys1979Classification}.
The other cases, where we can take $m=1$ without loss of generality since $G_{n,m} \cong G_{m,n} \cong G_{-m,-n}$, are precisely the soluble Baumslag--Solitar groups, with $G_{1, n}=\BS(1, n)$.
We call a group \emph{Gildenhuys-free} if it contains no Gildenhuys group as a subgroup.

The groups $G_{m, n}$ occur naturally as subgroups of $H\rtimes\mathbb{Z}$ groups, where $H$ is locally cyclic.
To see this, first note that $H$ is a subgroup of $(\mathbb{Q}, +)$ and for every automorphism $\phi\in\aut(H)$ there exist coprime integers $m,n\in\mathbb{Z}$ such that $\phi(a)=\frac{m}{n}a$ for all $a\in H$.
We therefore use the notation $\phi_{m, n}$ for the automorphism defined by $\phi(a)=\frac{m}{n}a$.
The groups $G_{m, n}$ occur as follows.

\begin{lemma}
\label{lem:GmnOccur}
Consider a semidirect product $\mathbb Q \rtimes\mathbb{Z}$, where the $\mathbb{Z}$-factor acts on $\mathbb Q$ by the automorphism $\phi_{m, n}$.
Then for all $x, y \in \mathbb Q$ with $x \neq 0$ and $p \in \mathbb{Z}$, the subgroup $K = \langle (x, 0), (y, p)\rangle$ of $\mathbb Q \rtimes\mathbb{Z}$ is isomorphic to the group $G_{m^p, n^p}$.
\end{lemma}

\begin{proof}
In general, given a semidirect product $N \rtimes Q$ we can modify the obvious splitting of the short exact sequence \[ 1 \to N \to N \rtimes Q \to Q \to 1\] by any crossed homomorphism $\psi \colon Q \to Z(N)$; this yields an automorphism $(n, q) \mapsto (n \psi(q), q)$ of $N \rtimes Q$.
For the present case of $\mathbb{Q} \rtimes \mathbb{Z}$, picking $\psi \colon \mathbb{Z} \to \mathbb{Q}$ to be the unique crossed homomorphism with
\[\psi(1) = - y \left(\sum_{i=0}^{p-1} \frac {m^i}{n^i}\right)^{-1} \]
 will give an automorphism sending $(y, p) \mapsto (0, p)$ and fixing $(x, 0)$.
Postcomposing with the automorphism $(a, r) \mapsto (a/x, r)$ gives the generating pair $(1, 0), (0, p)$.

The action of $(0, p)$ on $\mathbb{Q}$ is $\phi_{m, n}^p = \phi_{m^p, n^p}$ and thus we have $\mathbb{Z}[1 / mn] \subseteq \langle (1, 0), (0, p) \rangle \cap \mathbb{Q} \times \{0\}$.
Since $\mathbb{Z}[1 / mn]$ is invariant under this action, we indeed have the semidirect product $\mathbb{Z}[1 / mn] \rtimes_{\phi_{m^p, n^p}} \mathbb{Z} = G_{m^p, n^p}$ as claimed.
\end{proof}

\p{Poison-subgroup characterisation}
We now turn to our poison-subgroup characterisation of AH groups.
We start with a result which parallels Corollary \ref{corol:FirstClassification}.

\begin{proposition}
\label{prop:SecondClassification}
A torsion-free group $G$ is AH if and only if
\begin{enumerate}
\item\label{AHdef:CT2} $G$ is commutative-transitive, and
\item\label{AHdef:BS2} $G$ contains no Baumslag--Solitar group or Gildenhuys group as a subgroup.
\end{enumerate}
\end{proposition}

\begin{proof}
Every Baumslag--Solitar group contains a $\BS(1, k)$ subgroup for some $|k|\geqslant 1$ by \cref{lem:bs_contains_metabelian_bs}.
Therefore, every Baumslag--Solitar group and Gildenhuys group either is, or contains, a subgroup of the form $H\rtimes\mathbb{Z}$ with $H$ non-trivial locally cyclic.
Hence, if $G$ is AH then it satisfies (\ref{AHdef:BS2}), while it satisfies (\ref{AHdef:CT2}) by definition.

Suppose now that $G$ is torsion-free and (\ref{AHdef:CT2}) holds.
If $G$ is not AH then $G$ contains a subgroup $H\rtimes\mathbb{Z}$ with $H$ non-trivial torsion-free locally cyclic.
Any torsion-free locally cyclic group embeds into $\mathbb{Q}$, and so by Lemma \ref{lem:GmnOccur}, $G$ contains a solvable Baumslag--Solitar group or Gildenhuys group as a subgroup, so (\ref{AHdef:BS2}) does not hold.
The result follows from the contrapositive.
\end{proof}

Our poison-subgroup characterisation is as follows.

\begin{theorem}
\label{thm:SecondClassification}
A torsion-free group $G$ is AH if and only if $G$ is
$\BS$-free,
Gildenhuys-free, and
Adian-free.
\end{theorem}

\begin{proof}
As both AH and $\BS$-free groups contain no $\mathbb{Z}^2$-subgroups, Lemma \ref{lem:CTvsAdian} is applicable and so commutative-transitivity is equivalent to being Adian-free. The result now follows from Proposition \ref{prop:SecondClassification}.
%
%
\end{proof}

Theorems \ref{thm:wAHclassification} and \ref{thm:SecondClassification} differ only slightly, with the ``$G$ is weakly AH'' in Theorem \ref{thm:wAHclassification} being replaced by the ``$\BS$-free/Gildenhuys-free'' condition.
Corollary \ref{corol:FirstClassification} and Proposition \ref{prop:SecondClassification} also have this relationship.
It is therefore tempting to guess that a torsion-free group is weakly AH if and only if it is both $\BS$-free and Gildenhuys-free.
We have shown already that a weakly AH group is indeed both $\BS$-free and Gildenhuys-free, however the other direction is unclear and appears to be difficult; let us record it here as an open question.

\begin{question}
Let $G \neq \mathbb{Z}$ be an infinite torsion-free quotient of a Baumslag--Solitar group $\BS(m, n)$ with $|m| \neq |n|$.
Does $G$ contain a Baumslag--Solitar or Gildenhuys subgroup?
\end{question}

\p{Omitting one poison class}
We now prove that Theorem \ref{thm:SecondClassification} is sharp, in the sense that omitting any one of Baumslag--Solitar groups, Gildenhuys groups or Adian extensions as poison subgroups can yield groups which are not AH. We show in particular (as is indeed a logical necessity for this sharpness) that Baumslag--Solitar groups, Gildenhuys groups, and Adian extensions themselves demonstrate this.

First, we prove that Gildenhuys groups and Baumslag--Solitar groups do not embed into each other.

\begin{lemma}
\label{lem:BSvsG}
\leavevmode
\begin{enumerate}
\item Baumslag--Solitar groups are Gildenhuys-free.
\item Gildenhuys groups are $\BS$-free
\end{enumerate}
\end{lemma}

\begin{proof}
As Baumslag--Solitar groups are torsion-free one-relator groups, they have cohomological dimension $2$ and so their subgroups have cohomological dimension at most $2$.
However, Gildenhuys groups have cohomological dimension $3$ \cite[Theorem 4]{Gildenhuys1979Classification}, and the first assertion follows.

For the second assertion, let $G_{m, n}=\mathbb{Z}\left[1/mn\right]\rtimes\mathbb{Z}$ be a Gildenhuys group.
Suppose for the sake of contradiction that $G_{m,n}$ contains a Baumslag--Solitar subgroup and choose a standard generating pair $(\frac{a}b, p), (\frac{c}d, q)\in G_{m, n}$ for it.
Note first that the element $g:=(\frac{a}b, p)^q (\frac{c}d, q)^{-p}$ is contained in $\mathbb{Z}\left[1/mn\right]$.
If $g=(0, 0)$ then the intersection $\langle (\frac{a}b, p)\rangle\cap\langle(\frac{c}d, q)\rangle$ is non-trivial, which is not the case for the standard generators of a Baumslag--Solitar group.
If $g\neq(0, 0)$ then by Lemma \ref{lem:GmnOccur}, $\langle (\frac{a}b, p), g\rangle$ is isomorphic to the Gildenhuys group $G_{m^p, n^p}$, so $\langle (\frac{a}b, p), (\frac{c}d, q)\rangle$ contains a Gildenhuys group, contradicting (1).
The result follows.
\end{proof}

We deal with Adian extensions by considering commutative-tran\-si\-tiv\-i\-ty.
Note that the group $\BS(1, -1)$ is not commutative-transitive (as it is a non-abelian group with non-trivial center), and so in the following we exclude $G_{-1,1} \cong G_{1,-1} \cong \BS(1, -1)$.

\begin{lemma}
\label{lem:CT}
The groups $G_{m, n}$ with $\frac{m}{n} \neq -1$ are commutative-transitive.
\end{lemma}

\begin{proof}
If $\frac{m}{n}=1$ then $G\cong\mathbb{Z}^2$ is abelian and thus commutative-transitive, so from now on suppose $\frac{m}{n} \neq \pm 1$.
Then $G_{m, n} = \mathbb{Z}[1/mn] \rtimes \mathbb{Z} = \{ (a, p) \colon a \in \mathbb{Z}[1/mn], p \in \mathbb{Z}\}$ is naturally a subgroup of the affine group $\operatorname{Aff}(\mathbb{Q}) = \mathbb{Q} \rtimes \mathbb{Q}^\times$ under $(a, p) \mapsto (a, (\frac{m}{n})^p)$, so the next lemma completes the proof.
\end{proof}

\begin{lemma}
\label{lem:aff_is_ct}
Let $K$ be a field.
Then the affine group $\operatorname{Aff}(K) \cong K \rtimes K^\times$ is commutative-transitive.
\end{lemma}

We give a conceptual rather than computational proof.

\begin{proof}
We prove that every non-trivial element has abelian centraliser.
Consider the natural action of $\operatorname{Aff}(K)$ on $K$.
This action is sharply $2$-transitive (an affine map is determined uniquely by where it sends $0$ and $1$) so every non-trivial element has at most one fixpoint.
For a general group action $G \curvearrowright X$, the centraliser of an element $g \in G$ preserves the fixpoint set $X^g$ (as $hg = gh$ implies that if $g \cdot x = x$ then $g \cdot (h \cdot x) = h \cdot (g \cdot x) = h \cdot x$).
This means that if $g \in \operatorname{Aff}(K)$ acts with a (necessarily unique) fixpoint $x$ then its centraliser $C_{\operatorname{Aff}(K)}(g)$ coincides with the stabiliser $(\operatorname{Aff}(K))_x \cong K^\times$ and is abelian.
The map $x \mapsto a x + b$ is fixpoint free only if $a = 1$, and such an element can only commute with other fixpoint free elements or the identity, and these together comprise the translational subgroup $K$.
\end{proof}

We now prove our claim regarding omitting conditions in Theorem \ref{thm:SecondClassification} to get non-AH groups.
The first two statements yield groups which are in fact not even weakly AH.

\begin{theorem}
\label{thm:BSvsGvsA}
\leavevmode
\begin{enumerate}
\item\label{BSvsGvsA:1} Gildenhuys groups are torsion-free Adian-free $\BS$-free groups, but are not weakly AH.
\item\label{BSvsGvsA:2} Baumslag--Solitar groups $\BS(1, k)$ are torsion-free Adian-free Gildenhuys-free groups, but are not weakly AH.
\item\label{BSvsGvsA:3} Adian extensions are torsion-free $\BS$-free Gildenhuys-free groups, but are not AH.
\end{enumerate}
\end{theorem}

\begin{proof}
For (\ref{BSvsGvsA:1}) and (\ref{BSvsGvsA:2}), note that Gildenhuys groups and Baumslag--Solitar groups $\BS(1, k)$ are soluble (in fact metabelian) and so their finitely generated subgroups have no infinite torsion quotients, thus they are Adian-free.
Moreover, both are of the form $H \rtimes \mathbb{Z}$ for locally cyclic $H$ and are thus torsion-free and not AH, which in combination with being Adian-free implies they are not weakly AH by Theorem~\ref{thm:wAHclassification}.
Finally, recall that \cref{lem:BSvsG} states that Baumslag--Solitar groups are Gildenhuys-free and Gildenhuys groups are $\BS$-free.

For (\ref{BSvsGvsA:3}),
Adian extensions are torsion-free by definition.
Now, finitely generated non-cyclic subgroups of Adian extensions are themselves Adian extensions (noting that their image in the torsion quotient must be infinite, since torsion-free virtually cyclic groups are cyclic), and recalling that Gildenhuys groups and the Baum\-slag--Solitar groups $\BS(1, k)$ are Adian-free this means that Adian extensions are $\BS$-free and Gildenhuys-free.
Finally, Adian-extensions are not AH thanks to  Lemma~\ref{lem:TorsionGroups}.
\end{proof}

\p{Separating weakly AH and {\boldmath $\BS$-free}}
We can now prove the properness of the inclusion
\[\text{weakly AH}\subseteq \text{$\BS$-free}.\]
The counter-examples are Gildenhuys groups, which have finite cohomological, and hence geometric, dimension.
However, they are not of type $\mathtt F$ as they are not finitely presentable (or even ``{almost}'' finitely presentable) \cite[Theorem C]{Bieri1978almost}.

\begin{theorem}
\label{thm:coHom3Main}
There exist finitely generated $\BS$-free groups of cohomological dimension $3$ which are not weakly AH.
\end{theorem}

\begin{proof}
The groups in the theorem are the Gildenhuys groups, which have cohomological dimension $3$ \cite[Theorem 4]{Gildenhuys1979Classification} and, by Theorem \ref{thm:BSvsGvsA}.(\ref{BSvsGvsA:1}), are not weakly AH.
\end{proof}

\p{Separating the definitions}
We have now proved Theorem \ref{thm:inclusions:intro}.

\begin{theorem}[Theorem \ref{thm:inclusions:intro}]
\label{thm:inclusions}
We have the following chain of proper inclusions of classes of groups:
\[\text{AH}\subset\text{weakly AH}\subset\text{$\BS$-free}.\]
\end{theorem}

\begin{proof}
This is immediate from Theorems \ref{thm:weaklyAHvsAH} and \ref{thm:coHom3Main}.
\end{proof}

\p{Consequences for the complex group algebra}

As an example of a further algebraic consequence that controlling poison subgroups has, we now recall a theorem of Formanek regarding the Kaplansky idempotent conjecture and deduce a corollary for $\BS$-free groups.
(For brevity we consider only the complex group algebra but this implies the result for $K[G]$ where $K$ is any characteristic zero field.)

\begin{theorem}[{Formanek, \cite[Theorem 9]{Formanek73noetherian}}]
    \label{thm:formanek}
    Let $G$ be a torsion-free group.
    Suppose that for all $g \in G \smallsetminus \{1\}$ there are infinitely many primes $p$ such that for all $n \geqslant 1$, $g$ is not conjugate to $g^{p^n}$.
    Then the complex group algebra $\mathbb{C}[G]$ satisfies the idempotent conjecture, that is, it has no idempotents other than $0$ and $1$.
\end{theorem}

\begin{corollary}
    \label{corol:kaplansky_for_bs_free}
    Let $G$ be a torsion-free group.
    If $G$ is $\BS$-free, in particular if $G$ is AH, then $\mathbb{C}[G]$ satisfies the idempotent conjecture.
\end{corollary}

\begin{lemma}
    \label{lem:metabelian_bs_tf_quotient}
    Every non-trivial proper torsion-free quotient of $\BS(1, k)$ is isomorphic to $\mathbb{Z}$.
\end{lemma}

\begin{proof}
    If $k = 1$ this is an easy fact about torsion-free abelian groups and if $k = -1$ then this follows because $\BS(1,-1)$ has an index $2$ subgroup isomorphic to $\mathbb{Z}^2$ and the only torsion-free virtually cyclic group is $\mathbb{Z}$.

    Now suppose $k \neq \pm 1$.
    Write $A = \mathbb{Z}[\frac{1}{k}]$ so that $\BS(1, k) = A \rtimes \mathbb{Z}$ generated by $a = (1, 0)$ and $t = (0, 1)$.
    Let $N \vartriangleleft \BS(1, k)$ be a non-trivial normal subgroup.
    Then $N$ intersects $A$ non-trivially: if $g = (b, p) \in N \smallsetminus A$ then $[a, g] = (k^p - 1, 0) \in A \cap N$ and is non-trivial by assumption that $k \neq \pm 1$.
    But $A$ is locally cyclic so its proper quotients are locally finite cyclic, hence its image in any torsion-free quotient is trivial.
    Thus any torsion-free quotient of $\BS(1, k)$ factors through $\mathbb{Z}$.
\end{proof}

\begin{proof}[Proof of \Cref{corol:kaplansky_for_bs_free}]
    We show that $G$ satisfies the assumptions of \Cref{thm:formanek} in a strong way: for any $g \in G \smallsetminus \{1\}$, $g$ is not conjugate to $g^k$ for any $k \geq 2$ (in particular, for any $k = p^n$).
    Let $g, h \in G$ and $k \geq 2$.
    If $g^h = g^k$ then $\langle g, h \rangle$ is a quotient of $\BS(1, k)$ under $a \mapsto g, t \mapsto h$.
    If this is not an embedding of $\BS(1, k)$ then by torsion-freeness of $G$ and \Cref{lem:metabelian_bs_tf_quotient} we have that $\langle g, h \rangle$ is either trivial or $\mathbb{Z}$ and in either case $g = 1$.
\end{proof}


\section{Residually and fully residually (weakly) AH groups}
\label{sec:Residual}
Recall that the classes of AH and weakly AH groups are each closed under taking subgroups.
We now generalise this observation via residual properties.

\p{Fully residually AH groups}
A group $G$ is \emph{fully residually AH} if for every finite subset $S\subset G\smallsetminus\{1\}$ there exists an AH group $H_S$ and an epimorphism $\phi_S\colon G\twoheadrightarrow H_S$ such that for all $g\in S$, $\phi_S(g)$ is non-trivial.

In the following, the condition on $\mathbb{Z}^2$ subgroups is necessary, as $\mathbb{Z}^2$ itself is fully residually AH as it is fully residually $\mathbb{Z}$.

\begin{proposition}
\label{prop:residualAH}
Fully residually AH groups with no $\mathbb{Z}^2$ subgroups are themselves AH.
\end{proposition}

\begin{proof}
Let $G$ be a fully residually AH group.
Then $G$ is torsion-free as any torsion element of $G$ is contained in the kernel of every map from $G$ to any torsion-free, hence any AH, group.

Suppose $G$ is not commutative-transitive, so there exist $g, h, k\in G\smallsetminus\{1\}$ with $[g, h]=1$ and $[h, k]=1$ but $[g, k]\neq1$.
Take $S=\{g, h, k, [g, k]\}$.
Then under the associated map $\phi_S\colon G\rightarrow H_S$ we have $\phi_S([g, k])\neq1$, but also $\phi_S([g, k])=[\phi_S(g), \phi_S(k)]=1$ as $H_S$ is commutative-transitive, a contradiction.
Hence, $G$ is commutative-transitive.

Suppose that $G$ contains a subgroup $K=H\rtimes\mathbb{Z}$ with $H$ locally cyclic and non-trivial.
Let $t$ be the generator of the $\mathbb{Z}$-factor of $K$, and let $h$ be a non-trivial element of $H$ such that there exists $m, n\in\mathbb{Z}\smallsetminus\{0\}$ satisfying $t^{-1}h^mt=h^n$ (such an element exists as $H$ is locally cyclic, and as $t^{-1}Ht=H$).
As $G$ contains no $\mathbb{Z}^2$ subgroups, $[h, t]\neq1$.
Take $S=\{h, t, [h, t]\}$, and let $\phi_S\colon G\rightarrow H_S$ be the associated map.
Let $A$ be a maximal abelian subgroup of $H_S$ containing $\phi_S(h)$.
By Theorem \ref{thm:CSA}, $A$ is malnormal in $H_S$.
As $t^{-1}h^mt=h^n$ with $\phi_S(h)$ and $\phi_S(t)$ both non-trivial, this means that $\phi_S(t)\in A$.
Hence, $[\phi_S(h), \phi_S(t)]=\phi_S([h, t])=1$ a contradiction.
Hence, $G$ contains no subgroups $H\rtimes\mathbb{Z}$ with $H$ locally cyclic.
The result follows.
\end{proof}

\begin{example}
\label{ex:residualAH}
Let $\Gamma$ be a torsion-free hyperbolic group.
If $G$ has no $\mathbb{Z}^2$ subgroups and is a limit group over $\Gamma$, then $G$ is AH.
This follows from combining Proposition \ref{prop:residualAH} with the classification of limit groups over a hyperbolic group $\Gamma$ as the finitely generated fully residually $\Gamma$ groups \cite[Proposition 1.18]{Sela2009Diophantine}.
\end{example}

\p{Residually weakly AH groups}
Our result is stronger in the setting of weakly AH groups, as it relates to a residual, rather than fully residual, property:
A group $G$ is \emph{residually weakly AH} if for every non-trivial element $g\in G\smallsetminus\{1\}$ there exists a weakly AH group $H_g$ and a epimorphism  $\phi_g\colon G\twoheadrightarrow H_g$ such that $\phi_g(g)$ is non-trivial.

As in Proposition \ref{prop:residualAH}, the condition on $\mathbb{Z}^2$ subgroups is necessary as $\mathbb{Z}^2$ itself is residually weakly AH.

\begin{proposition}
\label{prop:residualWAH}
Residually weakly AH groups with no $\mathbb{Z}^2$ subgroups are themselves weakly AH.
\end{proposition}

\begin{proof}
Let $G$ be a residually weakly AH group.
As in the proof of Proposition \ref{prop:residualAH}, $G$ is torsion-free.

Suppose that we have $x\in G\smallsetminus\{1\}$ such that $x^m$ and $x^n$ are conjugate in $G$, $m, n\in\mathbb{Z}$.
Then there exists a map $\phi_x\colon G\rightarrow H_x$ where $H_x$ is weakly AH and $\phi_x(x)$ is non-trivial.
Conjugation is preserved in the image, so $\phi_x(x)^m$ and $\phi_x(x)^n$ are conjugate in $H_x$, and so as $H_x$ is torsion-free we have $|m|=|n|$.
The result follows.
\end{proof}

We therefore have the following.

\begin{corollary}
\label{corol:residualWAH}
Suppose that $G$ is a finitely generated group with no $\mathbb{Z}^2$ subgroups, and is locally indicable, or residually finite, or has cohomological dimension $2$.
If $G$ is residually weakly AH, then $G$ is AH.
\end{corollary}

\begin{proof}
By Proposition \ref{prop:residualWAH}, as $G$ is residually weakly AH it is itself weakly AH.
The result then follows by Theorem \ref{thm:LI} (when $G$ is locally indicable or residually finite) and Proposition \ref{prop:coHomBS} (when $G$ is of cohomological dimension $2$).
\end{proof}

In Example \ref{ex:residualAH} we basically considered finitely generated fully residually (torsion-free hyperbolic) groups.
We now wish to promote certain residually (torsion-free hyperbolic) groups to this class.
In particular, Remeslennikov proved that a residually free group with no $\mathbb{Z}^2$ subgroups is fully residually free \cite{Remeslennikov1989exists},
and we now extend this result to residually (torsion-free hyperbolic) groups.
To do this we assume residual finiteness; this assumption conjecturally holds, in the sense that it follows from the famous conjecture that all hyperbolic groups are residually finite.

\begin{corollary}
If $G$ is a finitely generated residually finite and residually (torsion-free hyperbolic) group with no $\mathbb{Z}^2$ subgroups, then $G$ is a fully residually (torsion-free hyperbolic) group.

In particular, if every hyperbolic group is residually finite then residually (torsion-free hyperbolic) groups with no $\mathbb{Z}^2$ subgroups are fully residually (torsion-free hyperbolic).
\end{corollary}

\begin{proof}
Let $G$ be as in the statement.
As $G$ is residually finite, by Corollary \ref{corol:residualWAH} it is AH, and so commutative-transitive.
The result then follows from a result Ciobanu, Fine and Rosenberger \cite[Corollary 3.2]{Ciobanu2016Classes}.

If every hyperbolic group is residually finite, then any residually hyperbolic group is residually finite (by composing the relevant mappings).
Therefore, the second assertion follows from the first.
\end{proof}

\section{Two-generator algebraically hyperbolic groups}
\label{sec:splittings}
An \emph{abelian splitting} of a group $G$ is a decomposition of $G$ as the fundamental group of a graph of groups with abelian edge groups; we define a \emph{locally cyclic splitting} analogously. Such splittings are \emph{essential} if no edge group has finite index in an adjacent vertex group.
In this section we prove Theorem \ref{thm:TwoGenJSJ}, which classifies the essential abelian splittings of two-generated torsion-free algebraically hyperbolic groups.

The theory of abelian splittings of torsion-free hyperbolic groups is a central topic in Geometric Group Theory, and for torsion-free hyperbolic groups these splittings are encoded in ``cyclic JSJ decompositions''.
In particular, these JSJ decompositions underlie the solution to the isomorphism problem for torsion-free hyperbolic groups \cite{Sela1995isomorphism} \cite{Dahmani2008isomorphism}, while a slight generalisation and adjustment plays the same role in the isomorphism problem for all hyperbolic groups \cite{Dahmani2011isomorphism}.

As torsion-free algebraically hyperbolic groups are CSA (Theorem \ref{thm:CSA}), finitely generated such groups have unique ``abelian JSJ-de\-com\-po\-si\-tions'' \cite{Guirardel2017JSJ}.
Therefore, Theorem \ref{thm:TwoGenJSJ} characterises such JSJ-decompositions of these groups.

\p{Free products with amalgamation and HNN-extensions}
If $A, B$ are groups with subgroups $C_A\leqslant A$ and $C_B\leqslant B$ which are isomorphic via the map $\phi\colon C_A\rightarrow C_B$, then we can form the \emph{free product of $A$ and $B$ amalgamating $C_A$ and $C_B$ across $\phi$}, which is the group with relative presentation
\[
A\ast_{C_A=C_B}B:=\langle A, B\mid c=\phi(c), c\in C_A\rangle.
\]
If $H$ is a group with subgroups $A$ and  $B$ that are isomorphic via the map $\phi\colon A\rightarrow B$, then we can form the \emph{HNN-extension of $H$ across $\phi$}, which is the group with relative presentation
\[
H\ast_{A^t=B}:=\langle H, t\mid t^{-1}at=\phi(a), a\in A\rangle.
\]
These are two of the fundamental constructions in Geometric Group Theory, and every graph of groups can be
thought of as encoding a collection of such operations, one for each edge, in the following sense: collapse all but one edge; the resulting graph of groups is either an amalgamated free product or an HNN extension.

We first consider free products with amalgamation, and then HNN-extensions.
As every graph of groups can be broken down into a collection of these constructions, these are the two key situations to consider. 

One of our main motivations for this paper was to extend the proof of Kapovich and Weidmann for ``RG groups'' as far as possible \cite{Kapovich1999structure}.
In particular, the result has been extended from cyclic to locally cyclic associated subgroups.
Therefore, our proof of Theorem \ref{thm:TwoGenJSJ} follows closely their proof, both in style and substance, but the more general setting means certain steps are approached slightly differently.
For example, the proofs of \ref{lem:KWcopyFPA}.(\ref{KWcopyFPA2}) and \ref{lem:KWcopyHNN}.(\ref{KWcopyHNN3}) are at the core the same inductive argument on ``block length'' as in Kapovich--Weidmann's 3.7.(2) and 3.8.(3), however they have been adapted to deal with locally cyclic associated subgroups and we have extracted the key Bass--Serre theoretic results into \Cref{lem:generating_amalgam_factor,lem:generating_hnn_factor} respectively for the sake of clarity and in the hope they are more useful to others in this standalone form.

\subsection{Free products with amalgamation}
We first consider the case of $G$ splitting as a free product with amalgamation over a locally cyclic subgroup.
%
%

\Cref{lem:KWcopyFPA} generalises Proposition 3.7 of \cite{Kapovich1999structure}; it incorporates more groups and allows splittings across locally cyclic subgroups.


\begin{proposition}
\label{lem:KWcopyFPA}
Let $G = A\ast_CB$ be a torsion-free AH group
which can be generated by two elements, and let $C$ be non-trivial and locally cyclic.
Assume that the splitting is non-trivial, that is $A \neq C$ and $B \neq C$. Then either $A$ or $B$ is locally cyclic. In the case that $A$ is locally cyclic (the case that $B$ is locally cyclic is analogous) we further get the following:
\begin{enumerate}
\item\label{KWcopyFPA1}
$C$ is malnormal in $B$ and has finite index in $A$, $|A:C|<\infty$.
\item\label{KWcopyFPA2}
There exist elements $a\in A$, $b \in B$ such that $G = \langle a,b\rangle$ and $B = \langle a^n, b\rangle$, where $n=|A:C|<\infty$.
\item\label{KWcopyFPA3}
The group $G$ is one-ended
if and only if $B$ is one-ended.
\end{enumerate}
\end{proposition}

As mentioned, we encapsulate part of the proof in the following lemma.

\begin{lemma}
    \label{lem:generating_amalgam_factor}
    Let $G = A \ast_C B$ be a free product with amalgamation.
    If subgroups $A_0 \leqslant A$ and $B_0 \leqslant B$ are such that $G = \langle A_0, B_0 \rangle$, and $C$ is normalised by $A_0$, then $B = \langle A_0 \cap C, B_0 \rangle$.
\end{lemma}

\begin{proof}[Proof of \cref{lem:generating_amalgam_factor}]
    Let $B_1 = \langle A_0 \cap C, B_0 \rangle \leqslant B$.
    We will prove that in fact $B_1 = B$.
    Let $b \in B$ be arbitrary.
    The key idea is to repeatedly use the normal form theorem for free products with amalgamation \cite[IV.2.6]{LS} to ensure words representing $b$ contain elements of $C$ and then use the fact that $C$ is normalized by $A_0$ to push them around and simplify the word.
    Since $A_0$ and $B_1 \geqslant B_0$ generate $G$, we can write $b$ as an alternating product of elements $x_i \in A_0$ and $y_i \in B_1$.
    Consider such a product of minimal possible length, which we write as either \[
        b = x_0 y_0 x_1 y_1 \dots \quad \text{or} \quad b = y_0 x_1 y_1 x_2 \dots
    \] according to whether the first factor of the product is some $x_0 \in A_0$ or some $y_0 \in B_1$.

    If the product is empty then $b = 1 \in B_1$.
    If the product is length $1$ then either $b = y_0 \in B_1$ and we are done or $b = x_0 \in A_0$ in which case, as $A \cap B = C$, we have $b \in A_0 \cap C \leqslant B_1$.
    So we now suppose the shortest product has length at least $2$ and work towards a contradiction.
    An immediate consequence of minimality is that $x_i \notin C$ for all $i$ as otherwise we could absorb it into an adjacent $y_{i-1}$ or $y_i \in B_1$ and obtain a shorter expression.
    We now claim that some $y_i \in C$.
    Recall that the normal form theorem of free products with amalgamation says in particular that a non-empty alternating product of elements from $A \smallsetminus C$ and $B \smallsetminus C$ cannot represent the trivial element.
    If all $y_i \notin C$ then there are two cases to consider.
    If the word representing $b$ starts with $x_0 \in A_0 \smallsetminus C$ then $b^{-1} x_0 y_0 \dots$ is an alternating product so the normal form theorem forces $b \in C$, whence $(b^{-1} x_0) \in A \smallsetminus C$ also and we get a contradiction with the alternating product $(b^{-1} x_0) y_0 \dots$.
    If the word representing $b$ starts with $y_0 \in B_1$ then $b^{-1} y_0 \in B$ and we run the same argument.
    Thus some $y_i \in C$.
    If the product has length $2$ then one factor is in $A \smallsetminus C$ and the other is in $C$, yielding $b \in A \smallsetminus C$, a contradiction.
    If the product has length greater than $2$, then we now use the fact that $C$ is normalised by $A_0$ to move a letter $y_i$ past an adjacent $x_i$ or $x_{i+1}$, collect together common terms and obtain a shorter product, a contradiction.
\end{proof}

\begin{proof}[Proof of \cref{lem:KWcopyFPA}]
We first prove that $C$ is malnormal in $A$ or $B$.
Suppose otherwise.
Then there exist $g\in A\smallsetminus C$, $h\in B\smallsetminus C$ such that $C^g \cap C \neq 1$ and $C^h \cap C \neq 1$.
Let $\overline{C} = C_G(C)$.
Combining \cref{lem:maxab_is_centraliser} and \cref{thm:CSA}, $\overline{C}$ is maximal abelian and hence malnormal.
Thus both $g \in \overline{C}$ and $h \in \overline{C}$, so $g$ and $h$ commute.
This is a contradiction, as they are contained in different free factors of $A\ast_CB$ and neither is contained in $C$.
Hence, $C$ is malnormal in $A$ or $B$.

Without loss of generality, assume that $C$ is malnormal in $B$. Then, by a result of Kapovich and Weidmann \cite[Corollary 3.2]{Kapovich1999TwoGen}, there exists a generating pair $(f,s)$ for $G$ such that $f \in A$ and $f^{z_1} \in C \smallsetminus\{1\}$ for some $z_1\in\mathbb{Z}$ and such that one of the following holds:
\begin{enumerate}[label=(\roman*)]
\item\label{KWTrees1}
$s\in B\smallsetminus C$;
\item\label{KWTrees2}
$s = ab$ with $a\in A\smallsetminus C$ and $b \in B\smallsetminus C$ and $a^{-1}f^{z_2}a\in C\smallsetminus\{1\}$ for some $z_2\in\mathbb{Z}$;
\item\label{KWTrees3}
$s = bab^{-1}$ with $a\in A\smallsetminus C$ and $b\in B\smallsetminus C$ and $a^{z_2} \in C\smallsetminus\{1\}$ for some $z_2\in\mathbb{Z}$.
\end{enumerate}
We first prove that if either of cases \ref{KWTrees2} or \ref{KWTrees3} hold then there exists a generating pair which satisfies \ref{KWTrees1}.

Suppose case \ref{KWTrees3} holds.
As $C$ is locally cyclic, $f$ and $a$ have a common power.
Hence, $\langle a, f\rangle$ has non-trivial centre and so is cyclic (by algebraic hyperbolicity).
Write $f'$ for the generator of $\langle a, f\rangle$, and note that as $a, f\in A$ we have $f'\in A$. Therefore, $G=\langle f', s\rangle$. As $s=bab^{-1}$ with $a\in\langle f'\rangle$, we further have that $G=\langle f', b\rangle$. Then the generating pair $(f', b)$ satisfies the conditions of case \ref{KWTrees1}.

Suppose case \ref{KWTrees2} holds.
Now, both $(a^{-1}f^{z_2}a)^{z_1}=a^{-1}f^{z_1z_2}a$ and $(f^{z_1})^{z_2}$ are contained in $C\leqslant \overline{C}$. As $\overline{C}$ is malnormal we have that $a\in\overline{C}$. Therefore, $f^{z_1}$ commutes with both $a$ and $f$, and so $f^{z_1}$ is contained in the centre of the group $\langle a, f\rangle$, and so $\langle a, f\rangle$ is cyclic by algebraic hyperbolicity. Write $f'$ for the generator of $\langle a, f\rangle$, and note that $f'\in A$. Therefore, $G=\langle f', s\rangle$. As $s=ab$ with $a\in\langle f'\rangle$, we further have that $G=\langle f', b\rangle$. Then the generating pair $(f', b)$ satisfies the conditions of case \ref{KWTrees1}.

Therefore, $G$ is generated by a pair $(f, s)$ satisfying the conditions of case \ref{KWTrees1}, so $f\in A$, $f^{z_1}\in C$ and $s\in B\smallsetminus C$.
The malnormal subgroup $\overline{C}=C_G(C)$ of $G$ is also a subgroup of $A$ as $C$ is malnormal in $B$ \cite[Theorem~1]{karrass1971free} (in the terminology of Karrass and Solitar's paper, $G$ is ``$1$-step malnormal'').
Assume for the sake of contradiction that $\overline{C}\lneq A$.
Then $G=A\ast_{\overline{C}}\overline{B}$ where $\overline{B}=\overline{C}\ast_CB$. Then $\overline{C}$ is malnormal in $A$ while its image is malnormal in $\overline{B}$. This is a contradiction as such a product cannot be generated by two elements
\cite[Lemmas 2 \& 3]{norwood1982every} (see also \cite[Corollary 3.2]{bleiler1998free}).
Hence, $A=\overline{C}$ and is locally cyclic.

We now prove that $C$ has finite index in $A$.
To see this, note that $C$ is a proper normal subgroup of $A$ (as $A=\overline{C}$ is abelian).
By the universal property of the amalgam $G = A\ast_CB$, there is a natural homomorphism $G \to A/C$ (under which $B$ has trivial image).
Thus $A/C$ is finitely generated, so since $A$ is locally cyclic, $A/C$ is a quotient of $\mathbb{Z}$.
It cannot be that $A/C \cong \mathbb{Z}$, as then $A \cong C \rtimes \mathbb{Z}$, contradicting algebraic hyperbolicity.
Thus $A / C$ is finite cyclic and in particular $[A:C] < \infty$.

The above proves that $C$ is malnormal in $A$ or $B$, and if it is malnormal in $B$ then it has finite index in $A$ and $A$ is locally cyclic.
Note that $B$ cannot also be locally cyclic, as $B$ contains a proper, non-trivial malnormal subgroup (namely $C$).
This therefore implies the initial part of the proposition: either $A$ or $B$ is locally cyclic, and if it is $A$ which is locally cyclic then (\ref{KWcopyFPA1}) holds.

We now prove (\ref{KWcopyFPA2}).
Recall that case \ref{KWTrees1} holds, so take $a:=f\in A$ and $b:=s\in B$ for our generating pair of $G$, and write $c:=a^n$ where $n=|A:C|$.
Note that the natural map $G \to A / C$ killing $B$ must send $a$ to a generator of $A / C \cong \mathbb{Z} / n$ (as the other element of the generating pair is sent to $1$), so $c$ is a generator of $\langle a \rangle \cap C$.
The result $B = \langle b, c \rangle$ now follows immediately from \Cref{lem:generating_amalgam_factor} setting $A_0 = \langle a \rangle$ and $B_0 = \langle b \rangle$.

To see that (\ref{KWcopyFPA3}) holds, assume first that $G$ is not one-ended.
Note that $B$ is non-abelian, as it contains a proper, non-trivial malnormal subgroup (namely $C$), and so $G$ is non-abelian.
Then $G$ is a free group of rank two since $G$ is non-abelian, torsion-free, two-generated and not one-ended -- this uses Stallings's theorem \cite{Stallings1971book} that a finitely generated group with more than one end splits as an amalgam or HNN extension over a finite subgroup and Grushko's theorem \cite{Grushko1940bases} that rank is additive under free products. It follows that $B < G$ is also free. Note that $B$ is not cyclic since it has a proper malnormal (locally) cyclic subgroup $C$. Thus $B$ is a non-abelian free group (in fact, of rank 2 by (\ref{KWcopyFPA2})), and therefore not one-ended.

Suppose now that $B$ is not one-ended. Then, as above, $B$ is a free group of rank two. Moreover, since $B = \langle b,c\rangle$, then $B$ is in fact free in $b,c$, that is $B = F(b,c)$, as finitely generated free groups are Hopfian. Then $C$ is necessarily cyclic, from which it follows that 
\[
G = \langle a \rangle\ast_{a^n=c} F(b,c) = F(a,b).
\]
 Thus, $G$ is a free group of rank two and is not one-ended.
\end{proof}

\subsection{HNN-extensions}
We next consider the case of $G$ being an HNN extension over locally cyclic subgroups.
\Cref{lem:KWcopyHNN} generalises Proposition 3.8 of \cite{Kapovich1999structure}; again, it incorporates more groups and allows splittings across locally cyclic subgroups.
If, as we shall do later, we assume that $A$ and $B$ are cyclic then one of them ($B$, say) is malnormal in $H$, so $C_H(B)=B$ is cyclic, but $C_H(A)$ may fail to be finitely generated.

If $G=H\ast_{A^t=B}$ is an HNN-extension of $H$, then Britton's Lemma gives a necessary condition for a word $W$ over $H$ and $t^{\pm1}$ to be trivial in $G$ \cite[Section IV.2]{LS}.
We use this condition in the following proof, and it is as follows: If $W=h_0t^{\epsilon_1}h_1\cdots h_{n-1}t^{\epsilon_n}h_n$, with $\epsilon_i=\pm1$ and $h_i\in H$, represents the trivial element of $G$ then there exists some index $i$ such that $t^{\epsilon_{i}}=-1$, $t^{\epsilon_{i+1}}=1$ and $h_i\in A$, or $t^{\epsilon_{i}}=1$, $t^{\epsilon_{i+1}}=-1$ and $h_i\in B$. (This condition means that $t^{\epsilon_{i}} h_{i}t^{\epsilon_{i+1}}$ collapses to an element of $A$ or $B$, as appropriate.)

\begin{proposition}
\label{lem:KWcopyHNN}
Let $G=H\ast_{A^t=B}$ be a torsion-free AH group which can be generated by two elements.
Assume that $A$ and $B$ are non-trivial locally cyclic subgroups of $H$. Then one of $A$ or $B$ is malnormal in $H$. In the case that $B$ is malnormal (the case that $A$ is malnormal is analogous) we get the following:
\begin{enumerate}
\item\label{KWcopyHNN1}
The group $G$ has a generating pair $(th,a)$, where $h \in H$ and $a\in C_H(A)$.
\item\label{KWcopyHNN2}
The group $H$ is not locally cyclic.
\item\label{KWcopyHNN3}
If $A$ is cyclic or if $H$ is finitely generated, the group $H$ can be generated by two elements, specifically $\langle a, h^{-1}b'h\rangle = H$ where $b'=t^{-1}a^nt\in B$ for some $n\in\mathbb{Z}\smallsetminus\{0\}$.
\item\label{KWcopyHNN4}
The group $G$ is one-ended if and only if $H$ is one-ended.
\end{enumerate}
\end{proposition}

We again extract part of the argument from \cite{Kapovich1999structure} and generalize it, both for clarity and ease of reuse.

\begin{lemma}
    \label{lem:generating_hnn_factor}
    Let $G = H \ast_{A^t = B}$ be a HNN extension and suppose that $H_0 \leqslant H$ is such that $G = \langle H_0, t \rangle$.
    Then $H = \langle H_0, A, B \rangle$.
\end{lemma}

\begin{proof}[Proof of \cref{lem:generating_hnn_factor}]
    Let $H_1 = \langle H_0, A, B \rangle \leqslant H$.
    We claim that $H_1 = H$.
    Let $h \in H$ be arbitrary.
    Note that $G = \langle H_1, t \rangle$ so we can write $h$ as a word \[
        h = h_0 t^{\epsilon_1} h_1 t^{\epsilon_2} \dots h_n
    \] with $h_i \in H_1$ and $\epsilon_i = \pm 1$.
    Consider such an expression with $n$ minimal.
    We now assume $h \notin H_1$ so that $n \geqslant 1$.
    Now, since $h^{-1} h_0 t^{\epsilon_1} \dots h_n = 1$, Britton's lemma tells us there is a \emph{pinch} subword of the form $t^{\epsilon_i} h_i t^{\epsilon_{i+1}}$ with either $\epsilon_i = -1, h_i \in A, \epsilon_{i+1} = 1$ or $\epsilon_i = 1, h_i \in B, \epsilon_{i+1} = -1$.
    In either case, the pinch represents an element of $H_1$ (specifically, of $B$ or $A$ respectively) and we can replace it with such an element to reduce word length, a contradiction.
\end{proof}

\begin{proof}[Proof of \cref{lem:KWcopyHNN}]
The group $G$ is CSA, by Theorem \ref{thm:CSA}, and hence so is the subgroup $H\leqslant G$.
Combining \cref{lem:maxab_is_centraliser} and \cref{thm:CSA}, the centralisers $C_G(A)$, $C_G(B)$ and $C_H(A)$, $C_H(B)$ are malnormal subgroups of $G$ and $H$ respectively.
We also have that all four subgroups are locally cyclic.

We now prove that one of $A$ or $B$ is malnormal in $H$.
Suppose otherwise.
Since $C_H(A)$ and $C_H(B)$ are malnormal in $H$, this means that $A\lneq C_H(A)$ and $B\lneq C_H(B)$.
Thus, there exist elements $a_1\in C_H(A)\smallsetminus A$ and $b_1\in C_H(B)\smallsetminus B$.
As $C_H(B)$ is locally cyclic and $G$ is torsion-free, there exists $n\in\mathbb{Z}$ such that $b_1^n\in B \smallsetminus \{1\}$.
Now $t b_1^n t^{-1} \in A \smallsetminus \{1\}$ and thus commutes with $a_1$.
As $H$ is commutative-transitive, we have that $a_1$ and $tb_1t^{-1}$ commute. However, by Britton's Lemma $a_1$ and $tb_1t^{-1}$ do not commute, a contradiction. Therefore, one of $A$ or $B$ is malnormal in $H$; suppose that $B$ is malnormal, so that $B=C_H(B)$.

We write $\overline{A}:=C_H(A)$.
We now prove that there exists no $h\in H$ such that $h^{-1}\overline{A}h\cap B\neq1$ (that is, $\overline{A}$ and $B$ are \emph{conjugacy separated in $H$}). Suppose such an $h\in H$ exists. Then there exists some element $c_0\in \overline{A}$ and integers $m, n\in\mathbb{Z}\smallsetminus\{0\}$ such that $th^{-1}c_0^mht^{-1}=c_0^n$. As $C_G(A)$ is malnormal in $G$ we have that $th^{-1}\in C_G(A)$, and as $C_G(A)$ is locally cyclic there exist integers $i, j\in\mathbb{Z}\smallsetminus\{0\}$ such that $(th^{-1})^ic_0^j=1$ in $G$. This is impossible by considering the natural surjection $G \to \mathbb{Z}$ under which $H$ has trivial image as that word gets sent to $i \neq 0$. Hence, by contradiction, $\overline{A}$ and $B$ are conjugacy separated in $H$.
Then $G$ has a generating pair $(th,a_0)$, where $h \in H$ and $a_0\in \overline{A}=C_H(A)$ \cite[Corollary 3.1]{Kapovich1999TwoGen}, and so (\ref{KWcopyHNN1}) holds.
Moreover, locally cyclic groups clearly do not contain two conjugacy separated non-trivial subgroups, and so (\ref{KWcopyHNN2}) holds. If $H$ is not finitely generated, we take $a:=a_0$. If $H$ is finitely generated, then we will take $a$ to be a specific root of $a_0$ as constructed in the following paragraph; this is necessary as the choice of $a$ connects to (\ref{KWcopyHNN3}).

We now prove (\ref{KWcopyHNN3}).
For notational convenience we do away with the conjugating element $h$ introduced above in the generating pair $(th, a_0)$: let $s = th$ and let $C = B^h$.
Then we can just as well write $G = H \ast_{A^s = C}$, as for example follows immediately from the presentations of the two HNN extensions.
We now apply \Cref{lem:generating_hnn_factor} to the HNN extension $H \ast_{A^s = C}$ with $H_0 = \langle a_0 \rangle$ to conclude that $H = \langle a_0, A, C \rangle$.
If $A$ is cyclic then so is $C$ and we conclude that $H$ is finitely generated anyway.
Choose a finite generating set for $H$ and words in $\{a_0\}, A, C$ representing the generators, letting the factors used be $a_1, \dots, a_k \in A$ and $c_1, \dots, c_l \in C$.
As $C^{s^{-1}} = A \leqslant C_H(A)$ and $a_0 \in C_H(A)$, which is a locally cyclic group, we can take $a \in C_H(A)$ as a generator of the cyclic group $\langle a_0, a_1, \dots, a_k, c_1^{s^{-1}}, \dots, c_l^{s^{-1}} \rangle$.
Then some power $a^n$ is a generator of the cyclic group $\langle c_1^{s^{-1}}, \dots, c_l^{s^{-1}} \rangle$ so that $s^{-1} a^n s$ is a generator of $\langle c_1, \dots, c_l \rangle$.
As $a^n \in A$, we set $b' = t^{-1} a^n t \in B$ and (recalling $s = th$) conclude that $\langle a, h^{-1} b' h \rangle = H$ as claimed.

To see \eqref{KWcopyHNN4}, suppose first that  $H$ is freely decomposable.
As $H$ is not locally cyclic, and so non-abelian, but can be generated by the pair $(a,b)$ with $b = h^{-1} b' h$, we have that $H$ is free with basis $(a,b)$, i.e.\ $H = F(a,b)$. Hence, $\overline{A}$ must be generated by $a$, $B$ must be generated by $b'=hbh^{-1}$, and so
\[
G =\langle H,s\mid s^{-1}a^ns = b \rangle=\langle a,b,s\mid s^{-1}a^ns = b \rangle=\langle a,s\mid \rangle= F(a,s)
\]
is a free group of rank two, and hence is not one ended.

Suppose now that $G$ is freely decomposable, i.e.\ that $G$ is a free group of rank two. Since $H$ is a two-generated non-cyclic subgroup of a free group $G$, $H$ itself is a free group of rank two, as required.
\end{proof}

\subsection{Abelian splittings}
We now combine the above two results to classify the abelian splittings of two-generated AH groups.
For this we need to precisely define the ``fundamental group of a graph of groups'', and related notions.

\p{Graphs of groups}
\label{gogs}
We use Serre's definition of a graph \cite{Trees}, so a graph $\Gamma$ has a set of vertices $V\Gamma$, while the edge set $E\Gamma$ consists of oriented ``half edges'', so each $e\in E\Gamma$ has an \emph{initial vertex} $o(e)$ and a \emph{terminal vertex} $t(e)$, and an inverse edge $\overline{e}$, with $e \neq\overline{e}$, $o(e) = t(\overline{e})$, and $t(e) = o(\overline{e})$.

A \emph{graph of groups} $\mathbf{\Gamma}$ consists of a connected graph $\Gamma$, a family of \emph{vertex groups} $\{\Gamma_v\mid v\in V \Gamma\}$, a family of \emph{edge groups} $\{\Gamma_e\mid e \in E\Gamma\}$ with $\Gamma_e = \Gamma_{\overline{e}}$ for all $e \in E\Gamma$, and a family of injections
\[
\alpha_e \colon \Gamma_e \longrightarrow \Gamma_{o(e)}
\]
for every $e \in E\Gamma$.

\p{Fundamental group of a graph of groups}
The ``graph of groups'' data $\mathbf{\Gamma}$ defines a group as follows.
The construction begins with a choice, but this does not change the isomorphism class of the group:
Let $T$ be a maximal subtree of $\Gamma$, the graph underlying $\mathbf{\Gamma}$.
The \emph{fundamental group of the graph of groups $\mathbf{\Gamma}$ with respect to the maximal subtree $T$} is the group
\[
\pi_1(\mathbf{\Gamma}, T):=(\ast_{v\in V\Gamma}\Gamma_v \ast F(E\Gamma))/N
\]
where $N$ is the normal closure of
\[
\{e\overline{e}\mid e \in E\Gamma\}\cup\{e\mid e\in T\}\cup\{e^{-1}\alpha_e(g)e = \alpha_{\overline{e}}(g)\mid e \in E\Gamma, g \in \Gamma_e\}.
\]
The reader is referred to Serre's book \cite{Trees} for more information about graphs of groups.

\p{Splitting, abelian splitting, essential splitting}
A \emph{splitting} of a group $G$ is a triple $(\mathbf{\Gamma},T,\psi)$ where $\mathbf{\Gamma}$ is a graph of groups, $T$ is a maximal subtree of $\Gamma$ and $\psi$ is an isomorphism
\[
\psi \colon \pi_1(\mathbf{\Gamma}, T) \longrightarrow G
\]
Below, we assume that $\psi$ is the identity map and identify $G$ with $\pi_1(\mathbf{\Gamma}, T)$.
A splitting $(\mathbf{\Gamma},T,\psi)$ of $G$ is called an \emph{abelian splitting} if all edge groups of $\mathbf{\Gamma}$ are abelian, and additionally is \emph{essential} if for every edge $e\in E\Gamma$ the subgroup $\alpha_e(\Gamma_e)$ is of infinite index in $A_{o(e)}$.

\p{Abelian splittings of AH groups}
We can now state and prove the main result of this section.
The result is especially strong when the fundamental group of the graph of groups carried by $T$ is finitely generated.
We do not expect this to happen in general, although it does happen when the splitting is cyclic; see below.

\begin{theorem}
\label{thm:TwoGenJSJ}
Let $G$ be a one-ended torsion-free AH group which can be generated by two elements. Let $(\mathbf{\Gamma},T, \psi)$ be an essential abelian splitting of $G$.
Then $(\mathbf{\Gamma},T, \psi)$ is a locally cyclic splitting, and
if the splitting is non-trivial then the rank of the fundamental group of the underlying graph $\Gamma$ is precisely $1$. Moreover, if the rank is $1$ and the fundamental group of the graph of groups carried by $T$ is finitely generated, then $T$ consists solely of a single vertex $v$, and
the vertex group $\Gamma_v$ is also a one-ended torsion-free AH group generated by two elements.
\end{theorem}

\begin{proof}
As $G$ is AH, abelian subgroups are locally cyclic. Therefore, $(\mathbf{\Gamma},T, \psi)$ is a locally cyclic splitting.
The proof
that the graph ${\Gamma}$ consists of either a single vertex with no edges or a single vertex with a single edge is essentially identical to the proof of \cite[Theorem 3.9]{Kapovich1999structure}, with the main differences being that the word ``cyclic'' should be replaced with ``locally cyclic'',
and we use \Cref{lem:KWcopyFPA,lem:KWcopyHNN} in place of their Propositions 3.7 and 3.8, respectively.
For completeness, we include our altered proof.

We will assume for convenience that $\psi = \mathrm{id}$. Let $E'$ denote the (possibly infinite) set of
positively oriented edges of ${\Gamma}$ lying outside the maximal subtree $T$.
We show first that $|E'| \leqslant 1$: It is clear that the map that quotients out the vertex groups ${\Gamma}_v$ for all $v\in V {\Gamma}$ defines a surjective homomorphism
\[\psi \colon \pi_1(\mathbf{\Gamma},T) \rightarrow F(E')\]
to the free group with basis $E'$. Hence the free group $F(E')$ is a homomorphic image of $G$. Since $G$ is a one-ended two-generated group, $G$ is a proper quotient of the free group of rank two $F_2$. Hence $F(E')$ is a proper quotient of $F_2$. This is impossible if $|E'|\geqslant 2$, since $F_2$ is Hopfian. Thus $|E'| \leqslant 1$.

We now write $n:=|E'|$, noting that it is $0$ or $1$.

Let us suppose $n=0$, and that $\Gamma=T$ has at least two vertices.
Then $T$ contains an edge $e$ with $v_0 = \iota(e) \neq \tau(e) = v_1$, and  $v_0,v_1 \in VT$. Thus $T \smallsetminus e$ consists of two connected components, $F_0$, containing $v_0$ and $F_1$, containing $v_1$. Let $\mathbb{F}$ be the graph of groups obtained from $T$ by collapsing $F_0$ and $F_1$. Clearly the underlying graph $F$ of $\mathbb{F}$ is just the edge $e$. It follows that $G \cong \pi_1(\mathbb{F},F)$, i.e.\ $G \cong K \ast_{C_0 = C_1} L$ where $C_0 = C_1$ is a locally cyclic group that is of infinite index in $K$ and $L$. By  \Cref{lem:KWcopyFPA} this is not possible since $G$ is 2-generated. Thus $T$ cannot have more than one vertex, so the assertion for $n=0$ is proved.

Now let us suppose that $n=1$.
 The graph $T$ is endowed with the structure of a graph of groups $\mathbf{T}$ inherited from $\mathbf{\Gamma}$. Let $H = \pi_1(\mathbf{T},T)$, and assume that $H$ is finitely generated. Since $n = 1$, we have $G =\langle H,e\mid e^{-1}C_1e = C_2 \rangle$ where $C_1$ and $C_2$ are locally cyclic subgroups of infinite index in $H$.
Hence by \Cref{lem:KWcopyHNN} the group $H$ is two-generated.
Finally, we may argue as in the $n=0$ case, so replacing $G$ with $H$ in the above paragraph, to get that $T$ cannot have more than one vertex. The result follows.


The vertex group $H$ is a torsion-free AH group, as these properties are preserved when taking subgroups, and is one-ended by \Cref{lem:KWcopyHNN}.
\end{proof}

\subsection{Cyclic splittings}
If we assume that the splitting in \cref{thm:TwoGenJSJ} is a cyclic splitting, we get a stronger structural result.
For this and other reasons, it is interesting and useful to understand when the splitting of \cref{thm:TwoGenJSJ} is in fact a cyclic splitting.
We now give the strengthened version of \cref{thm:TwoGenJSJ} for cyclic splittings, as well as two results, each of which concludes that an essential abelian splitting is in fact a cyclic splitting.

\begin{corollary}
Let $G$ be a one-ended torsion-free AH group which can be generated by two elements.
Let $(\mathbf{\Gamma},T, \psi)$ be an essential cyclic splitting of $G$.
The one of the following holds:
\begin{enumerate}
\item The graph $\Gamma$ consists of a single vertex with no edges.
\item The graph $\Gamma$ consists of a single vertex $v$ and a single edge $e$. In this case the vertex group $\Gamma_v$ is also a one-ended torsion-free AH group which can be generated by two elements.
\end{enumerate}
\end{corollary}

\begin{proof}
By \Cref{thm:TwoGenJSJ}, if the splitting $(\mathbf{\Gamma},T, \psi)$ is non-trivial then $G$ splits as an HNN-extension $H\ast_{A^t=B}$ where $A$ and $B$ are edge groups of the original splitting and where $H$ is the fundamental group of the graph of groups carried by $T$.
By \Cref{lem:KWcopyHNN}, as $A$ and $B$ are cyclic the group $H$ is two-generated.
The result now follows from \Cref{thm:TwoGenJSJ}.
\end{proof}

We first assume that $G$ has maximum-rank abelianisation.

\begin{corollary}
Let $G$ be a one-ended torsion-free AH group which can be generated by two elements, and additionally assume that the abelianisation is $\mathbb{Z}^2$.
Let $(\mathbf{\Gamma},T, \psi)$ be an essential abelian splitting of $G$.
Then $(\mathbf{\Gamma},T, \psi)$ is a cyclic splitting.
\end{corollary}

\begin{proof}
By Theorem \ref{thm:TwoGenJSJ}, if the splitting $(\mathbf{\Gamma},T, \psi)$ is non-trivial then $G$ splits as an HNN-extension $H\ast_{A^t=B}$ where $A$ and $B$ are edge groups of the original splitting and where $H$ is the fundamental group of the graph of groups carried by $T$.
Suppose $A$ and $B$ are non-cyclic.
As $G$ is AH, the centralisers $C_H(A)$ and $C_H(B)$ are locally cyclic, and hence also non-cyclic locally cyclic.
This means that any homomorphisms $C_H(A)\to\mathbb{Z}^2$ and $C_H(B)\to\mathbb{Z}^2$ must have trivial image, and so as $G$ has $\mathbb{Z}^2$ abelianisation we have that $C_H(A)$ and $C_H(B)$ are contained in the derived subgroup of $G$.
By \Cref{lem:KWcopyHNN}, $G$ has generating pair $(th, a)$ where $h\in H$ and $a\in C_H(A)$.
As $a$ is in the derived subgroup of $G$, it follows that $G$ has cyclic abelianisation.
This is a contradiction, and so we conclude that $A$ and $B$ are cyclic, and hence that the splitting is cyclic.
\end{proof}

We next combine finite presentability with properties of the splitting itself.
Here, we need a preliminary lemma.

\begin{lemma}
\label{lem:HNNfgsubgroups}
Suppose a finitely presented group $G$ splits as an HNN-extension $H\ast_{A^t=B}$, with $H$ finitely generated. Then $A$ and $B$ are finitely generated.
\end{lemma}

\begin{proof}
Write $\phi\colon A\rightarrow B$ for the isomorphism associated to the HNN-extension $G=H\ast_{A^t=B}$.
Suppose $A$ is not finitely generated, with generating set $\mathcal{A}$, and let $\langle \mathbf{x}\mid\mathbf{r}\rangle$ be a presentation of $H$, with $\mathbf{x}$ finite.
Then $G$ has presentation:
\[
\mathcal{P}=\langle \mathbf{x}, t\mid \mathbf{r}, t^{-1}at=\phi(a), a\in\mathcal{A}\rangle.
\]
As $G$ is finitely presented, a finite subset $\mathbf{s}\subset\{\mathbf{r}, t^{-1}at=\phi(a), a\in\mathcal{A}\}$ of the relators has the same normal closure in $F(\mathbf{x}, t)$ as the whole set of relators, $\langle\!\langle\mathbf{s}\rangle\!\rangle=\langle\!\langle\mathbf{r}, t^{-1}at=\phi(a), a\in\mathcal{A}\rangle\!\rangle$.
Hence, there exists a finite subset $\mathcal{A'}\subset\mathcal{A}$ such that
\[
\langle\!\langle\mathbf{r}, t^{-1}at=\phi(a), a\in\mathcal{A}\rangle\!\rangle=\langle\!\langle\mathbf{r}, t^{-1}at=\phi(a), a\in\mathcal{A'}\rangle\!\rangle.
\]
The presentation $\langle \mathbf{x}, t\mid\mathbf{r}, t^{-1}at=\phi(a), a\in\mathcal{A'}\rangle$ is an HNN-extension $H\ast_{A_1^t=B_1}$ where $A_1, B_1$ are proper subgroups of $A, B$ respectively.
Therefore, for any element $a\in A\smallsetminus A_1$, we have that $t^{-1}at\phi(a)^{-1}\neq1$ by Britton's Lemma.
On the other hand, in the infinite presentation $\mathcal{P}$ we have that $t^{-1}at\phi(a)^{-1}=1$.
As both sets of relators have the same normal closure in $F(\mathbf{x}, t)$, this is a contradiction.
\end{proof}

We now consider a given splitting, and assume that vertex groups are finitely generated.

\begin{corollary}
Let $G$ be a one-ended torsion-free AH group which can be generated by two elements, and additionally assume that $G$ is finitely presented.
Let $(\mathbf{\Gamma},T, \psi)$ be an essential abelian splitting of $G$, such that every vertex stabiliser is finitely generated.
Then $(\mathbf{\Gamma},T, \psi)$ is a cyclic splitting.
\end{corollary}

\begin{proof}
The assumption that every vertex stabiliser is finitely generated implies that the graph of groups carried by $T$ is finitely generated.
Thus, by Theorem \ref{thm:TwoGenJSJ}, if the splitting is non-trivial then it corresponds to an HNN-extension $H\ast_{A^t=B}$ with $H$ finitely generated.
Therefore, by Lemma \ref{lem:HNNfgsubgroups}, $A$ and $B$ are finitely generated, and so cyclic, as required.
\end{proof}

\bibliographystyle{amsalpha}
\bibliography{BibTexBibliography}
\end{document}